\documentclass{article}
\usepackage[utf8]{inputenc}
\usepackage{amsmath,amsthm,amssymb}
\usepackage{color}

\usepackage{ulem}

\theoremstyle{plain}
\newtheorem{defi}{Definition}
\newtheorem{theorem}[defi]{Theorem}
\newtheorem{proposition}[defi]{Proposition}
\newtheorem{remark}[defi]{Remark}
\newtheorem{corollary}[defi]{Corollary}
\newtheorem{lemma}[defi]{Lemma}
\newtheorem{example}[defi]{Example}

\newcommand{\diag}{{\rm diag\,}}
\newcommand{\rk}{{\rm rk\,}}

\newcommand{\Ac}{{\mathcal A\,}}

\newcommand{\R}{{\mathbb{R}\,}}
\newcommand{\N}{{\mathbb{N}\,}}

\usepackage{xcolor}
\usepackage{colortbl}
\definecolor{hage}{rgb}{0.4,0.6,1}

\title{Schur reduction of trees and extremal entries of the Fiedler vector}

\author{H.\ Gernandt\thanks{Institute of Mathematics, TU Ilmenau, Weimarer Stra\ss e 25, 98693 Ilmenau, Germany ({\tt hannes.gernandt@tu-ilmenau.de}).}\and J.\ P.\ Pade\thanks{Institute of Mathematics, Humboldt-University of Berlin, Unter den Linden 6, 10099 Berlin, Germany}}
\date{\today}

\usepackage{graphicx}

\begin{document}

\maketitle

\begin{abstract}
We study the eigenvectors of  Laplacian matrices of trees. The Laplacian matrix is reduced to a tridiagonal matrix using the Schur complement. This preserves the eigenvectors and allows us to provide fomulas for the ratio of eigenvector entries.
We also obtain bounds on the ratio of eigenvector entries along a path in terms of the eigenvalue and Perron values. The results are then applied to the Fiedler vector. Here we locate the extremal entries of the Fiedler vector and study classes of graphs such that the extremal entries can be found at the end points of  the longest path. 
\end{abstract}

%

\section{Introduction}
For a simple undirected unweighted graph $G$ with vertices $V(G)=\{v_1,\ldots,v_n\}$ and edges $E(G)$ the \textit{graph Laplacian} is given by $L(G):=D-A\in\mathbb{R}^{n\times n}$ where $D$ is a diagonal matrix containing the degrees of the vertices and $A$ is the adjacency matrix of the graph.

Since the seminal papers \cite{Fiedler1973,Fiedler1975} by M.\ Fiedler in the 1970s, the analysis of graph Laplacians has attracted a great deal of attention \cite{Grone1990,Chen2012,Zhang2011,Guo2006, Cvetkovic1980, Mohar1992}.
It is well known that $L(G)$ is positive semi-definite with eigenvalues $0=\lambda_1\leq\lambda_2\leq\ldots\leq\lambda_n$.

A particular focus lies on the problem of establishing a connection between algebraic properties of the graph Laplacian and the topology of the underlying graph.
For example, if $G$ is connected then $\lambda_1=0$ is a simple eigenvalue. Many eigenvalue bounds have been established in dependence of the graph topology for the other eigenvalues \cite{Chen2012} and especially for the smallest non-zero eigenvalue $\lambda_2$, the so-called \textit{algebraic connectivity} usually denoted by $a(G)$, see \cite{Abreu2007} for an overview. If $a(G)$ is a simple eigenvalue then the associated eigenvector is called \textit{Fiedler vector} in honour to M.\ Fiedler \cite{Fiedler1973}.

However, apart from the original results from M.\ Fiedler, very few is known about the Fiedler vector and its connection to topological properties of the underlying graph, see \cite{Biyikoglu2007,Merris1998,Griffing2013}. Not only is a deeper knowledge of this relation of a theoretical interest, it is also of great importance for many applications. In networks of diffusively coupled elements it was shown that the dynamical impact of connecting two nodes through an additional edge is closely related to the corresponding entries in the Fiedler vector \cite{Pade2015,Pade2016}. Furthermore, the Fiedler vector plays a central role in random walks on graphs, and applications to community detection \cite{Lovasz1993,Estrada2011}.

In the above applications, the extremal values of the Fiedler vector are of special interest. For instance, in networks of diffusively coupled elements, they correspond to the nodes which have the greatest impact on the dynamics when connected through an additional edge.

In 1974 it was hypothesized by J.\ Rauch that for a somewhat generic choice of initial conditions, the extremal values of the solution to the heat equation are attained at the boundary of the considered domain \cite{Banuelos1999}. This hypothesis turned out not to be true for certain domains \cite{Burdzy1999}. Later on, it was found that the discrete analogue of this hypothesis plays an important role in medical imaging processing \cite{Chung2011,Shen2010}. In \cite{Chung2011} it was hypothesized for trees that the extremal values of the Fiedler vector are attained at the two vertices which are connected by the longest path in the tree, or in other words, at the most distant vertices. It was only shown for a path though. And it was in 2013 that a counter-example among trees was found: the Fiedler rose \cite{Evans2011,Lefevre2013}, see also \cite{Abreu2016}. Since then, to our knowledge no progress has been made in verifying the hypothesis for a nontrivial class of trees. 

In this article, we investigate the structure of eigenvectors of trees. Here we use a graph reduction technique based on Schur complements which is similar to the well known Kron reduction \cite{Doerfler2012a,Griffing2009,Griffing2013}. However, our technique preserves the eigenvectors after reduction. This allows us to obtain formulas for the ratios of eigenvector entries.
We also provide upper and lower bounds for the ratios of the eigenvector entries along paths in the tree and we prove the hypothesis for a class of trees.

The article is structured as follows. In Section \ref{sec:pre}, we recall basic notions from graph theory and linear algebra. The Schur reduction is introduced in Section \ref{sec:schur} and its properties are studied. In Section \ref{sec:ratios} we provide formulas for the entries of Laplacian eigenvectors in terms of the Schur complement. 
In Section \ref{sec:fiedler} we apply our results to the Fiedler vector. First, it is shown that the extremal entries of the Fiedler vector are located at the pendant vertices of the tree. Later on, we give conditions to find those pendant vertices where the extremal entries are located. 
Furthermore, we study generalizations of caterpillar trees, where we can show that the extremal entries of the Fiedler vector are located at the endpoints of the longest path. In this context, we also discuss the Fiedler rose from \cite{Evans2011,Lefevre2013}.
Later in Section \ref{sec:evalue-bounds}, we obtain bounds on the ratios of eigenvector entries along paths that depend only on the eigenvalues. Finally, in Section \ref{sec:locExtr} we identify local extrema of the Fiedler vector in an even larger class of trees.

\section{Notations and Preliminaries}
\label{sec:pre}
In this section, we recall some notions from graph theory and linear algebra that we will use throughout the article.

For a graph $G$ we denote by $V(G)$ and $E(G)$ its \textit{set of vertices} and \textit{edges}, respectively. Each edge $e\in E(G)$ connects two vertices, say $v,w\in V(G)$ and we also write $vw$ instead of $e$. In this case, we say that $v$ and $w$ are \textit{adjacent} and that $e$ is \textit{incident} with $v$ and $w$.

The \textit{degree} of a vertex $v$, i.e.\ the number of incident edges, is denoted by $\deg_G(v)$. A vertex $v\in V(G)$ with $\deg_G(v)=1$ is called \textit{pendant vertex}. 

Let $G$ be a connected graph. The \textit{distance} $d(v,w)$ between two vertices $v,w\in V(G)$ is the number of edges in the shortest path between $v$ and $w$. The \textit{diameter} of $G$ is then given by
\[
d(G):=\max_{v,w\in V(G)} d(v,w).
\]
The \textit{path} with $n$ vertices is denoted by $P_n$. We also study \textit{star graphs}, i.e. trees $T$ with diameter $d(T)=2$, which we denote by $S_n$, where $n$ is the number of vertices. The unique vertex in $S_n$ with $n\geq 3$ which is not a pendant vertex is called \textit{center} of $S_n$.

We recall some definitions from linear algebra. For this sake, we consider a matrix $M\in\R^{n\times n}$.
We denote by $\sigma(M)$ the \textit{spectrum}, i.e.\ the set of eigenvalues, of $M$. Furthermore, $\|M\|:=\sup_{\|x\|=1}\|Mx\|$ is the \textit{spectral norm}. If $M$ is symmetric, then $\|M\|$ equals the eigenvalue with maximum modulus, i.e.\ the \textit{spectral radius} $\rho(M)$. Recall that the row sum norm is given by $\|M\|_{\infty}:=\max\limits_{1\leq i\leq n}\sum_{j=1}^n|m_{ij}|$ with $M=(m_{ij})_{i,j=1}^n\in\R^{n\times n}$.
For a symmetric matrix $M$ with nonnegative eigenvalues, we denote by $\lambda_{\min}(M)$ the smallest element in $\sigma(M)$ and if $M$ is invertible we have  $\lambda_{\min}(M)=\|M^{-1}\|^{-1}$.

Recall that for a block matrix $\Ac=\begin{pmatrix}A& B \\ B^\top& C\end{pmatrix}\in\R^{n\times n}$ with  $C\in\R^{r\times r}$ invertible, the \textit{Schur complement} with respect to the lower diagonal block $C$ is given by
\begin{equation*}
(\Ac/C):=A-B C^{-1}B^\top.
\end{equation*}

In the following we study the spectral properties of the \textit{graph Laplacian} $L(G)=D-A$
where $D=\diag(\deg_T(v_i))_{i=1}^n$ is the diagonal matrix of vertex degrees and

 $A=(a_{ij})_{i,j=1}^n$ is the adjacency matrix given by
\[
a_{ij}=\begin{cases} 1,& \text{if $e=v_iv_j\in E(G)$,}\\ 0,& \text{if $e=v_iv_j\notin E(G)$.}\end{cases}
\]
Since there is a natural labelling of the entries of the eigenvectors $(x_i)_{i=1}^n$ using the vertex set  $V(G)$, we will also write $x_{v_i}$ instead of $x_i$.

The associated \textit{reduced Laplacian} $L_{v_i}(G)\in\mathbb{R}^{(n-1)\times (n-1)}$ is obtained by deleting the $i$-th row and the $i$-th column of $L(G)$.

It is a regular matrix which is also known under the names of grounded Laplacian matrix \cite{Miekkala1993,Pirani2014,Xia2017} and Dirichlet Laplacian matrix \cite{Barooah2006}.
By the matrix-tree-theorem \cite{Molitierno2012},  $\det(L_{v_i}(G))$ is the number of spanning trees of $G$, so $\det(L_{v_i}(G))>0$, i.e. $L_{v_i}(T_i)$ is invertible.
We also consider the \textit{doubly reduced} Laplacian $L_{v_i,v_j}(G)\in\R^{(n-2)\times(n-2)}$ which is the matrix obtained from $L(G)$ by deleting simultaneously the rows and columns with index $i$ and $j$.

Finally, we denoted by $\N$ the set of natural numbers including zero.

\section{Schur reduction of trees}\label{sec:schur}

In this section, we present a reduction technique for the graph Laplacian that is based on the Schur complement.

Let $v_1v_2\ldots v_k$ be a path in an arbitrary tree $T$. Then to each vertex $v_i$ there is an associated unique maximal tree $T_i$ attached to it with $V(T_i)\cap\{v_1,\ldots,v_k\}=v_i$ (see Figure \ref{dergraph}) such that there are no edges between $T_i$ and $T_j$ for all $i\neq j$ except for $v_iv_{i+1}$. We say that $T_i$ \textit{is associated with} $v_i$.
\begin{figure}
    \centering
    \includegraphics[scale=0.5]{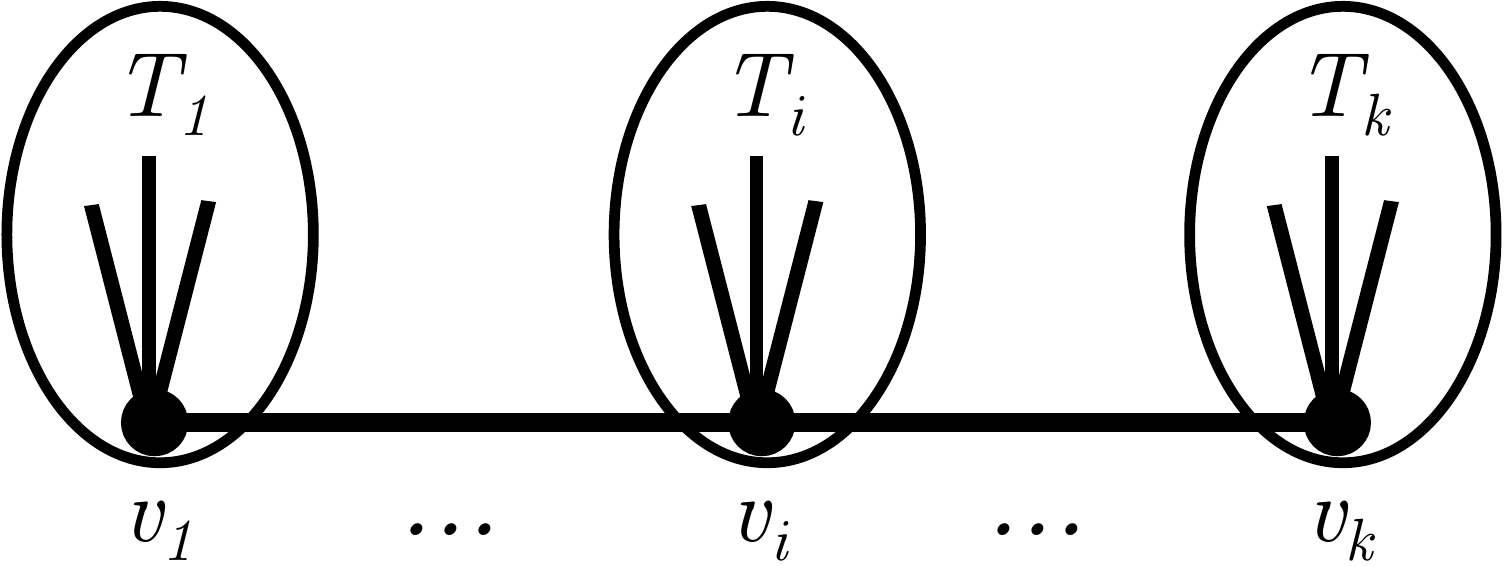}
    \caption{In a tree $T$ we select a path $v_1\ldots v_k$. Then for each $v_i$ on this path there is a unique maximal tree $T_i$ with
    $V(T_i)\cap\{v_1,\ldots,v_k\}=\{v_i\}$. Since $T$ is a tree, there are no edges between $T_i$ and $T_j$ for $i\neq j$.}
    \label{dergraph}
\end{figure}

Therefore, after a suitable relabelling of vertices, the graph Laplacian of $T$ can be written with the reduced Laplacians $L_{v_i}(T_i)$ in the form
\begin{align}
\label{lt}
L(T)=\begin{bmatrix} \deg_T(v_1)&-1&&&f_1^\top&\\ -1&\ddots&\ddots&&&\ddots& \\ &\ddots&\ddots&-1&&&\\ &&-1&\deg_T(v_k)&&&f_k^\top\\ f_1&&&&L_{v_1}(T_1)&&\\ &\ddots&&&&\ddots&\\ &&&f_k&&&L_{v_k}(T_k)\end{bmatrix}
\end{align}
where $f_i\in\R^{|V(T_k)|-1}$ is a vector with entries $-1$ if $v_i$ and the vertex in $T_i$ corresponding to the entry are adjacent or $0$ if $v_i$ is not adjacent with the corresponding vertex, i.e.\
\[
L(T_i)=\begin{bmatrix} L_{v_i}(T_i) & f_i\\ f_i^\top & \deg_T(v_i)     \end{bmatrix}.
\]

Note that for $\mathcal{A}=L(G)$ with a suitable reduced Laplacian $C=L_{v_i}(G)$ the Schur complement $(\Ac/C)$ is also called a Kron reduction of the graph $G$, see \cite{Doerfler2012a}.

We now investigate the Schur complement with respect to the lower diagonal block $\diag(L_{v_i}(T_i)-\lambda)_{i=1}^k$, which is given for  $\lambda\notin\sigma(L_{v_i}(T_i))$ for all $i=1,\ldots,k$ by

\begin{align}
\begin{split}
\label{tridiag}
S_{T_1,\ldots,T_k}(\lambda):=
((L(T)-\lambda)/\diag(L_{v_i}(T_i)-\lambda)_{i=1}^k) =\begin{pmatrix}s_{T_1}(\lambda) & -1 & & \\ -1& \ddots& \ddots & \\ &\ddots &\ddots &-1\\ &&-1& s_{T_k}(\lambda)\end{pmatrix}
\end{split}
\end{align}
where for $\lambda\notin\sigma(L_{v_i}(T_i))$ the function $s_{T_i}(\lambda)$ is given by
\begin{align}
\begin{split}
\label{def:fTi}
s_{T_i}(\lambda)&:=\deg_T(v_i)-\lambda-f_{T_i}(\lambda),\\
f_{T_i}(\lambda)&:=f_i^\top(L_{v_i}(T_i)-\lambda)^{-1}f_i.
\end{split}
\end{align}

In the theorem below, we relate the eigenvectors of $L(T)$ and $S_{T_1,\ldots,T_k}(\lambda)$. This will enable us to compare and estimate entries of the Fiedler vector of $L(G)$ in the subsequent sections.
\begin{theorem}
\label{thm:schur}
Let $T$ be a tree with $L(T)$ of the form \eqref{lt} and let $\lambda\notin\sigma(L_{v_i}(T_i))$ for all $i=1,\ldots,k$ then the following holds.
\begin{itemize}
    \item[\rm (a)] $\lambda\in \sigma(L(T))$ if and only if $\ker S_{T_1,\ldots,T_k}(\lambda)\neq \{0\}$.
    \item[\rm (b)] $(x_1,\ldots,x_k,y_1^\top,\ldots,y_k^\top)^\top\in\ker(L(T)-\lambda)$ if and only if  $(x_1,\ldots,x_k)^\top \in \ker S_{T_1,\ldots,T_k}(\lambda)$ and
    \begin{align}
    \label{yeq}
    y_i=-(L_{v_i}(T_i)-\lambda)^{-1}f_ix_i,\quad i=1,\ldots,k.
    \end{align}
    \item[\rm (c)] We have $\dim\ker S_{T_1,\ldots,T_k}(\lambda)\leq 1$, hence $(x_1,\ldots,x_k)^\top$ in (b) is unique up to scaling and  $\dim\ker(L(T)-\lambda)\leq 1$. Furthermore,  every eigenvector for $\lambda\in\sigma(L(T))$ satisfies $x_1\neq 0$, $x_k\neq 0$.
\end{itemize}
\end{theorem}
\begin{proof}
We abbreviate
\[
A:=\begin{pmatrix} \deg_T(v_1)-\lambda&-1 &&\\-1&\ddots &\ddots &\\ &\ddots&\ddots&-1\\&&-1&\deg_T(v_k)-\lambda
\end{pmatrix},\quad B:=\diag(f_i^\top)_{i=1}^k,\quad C:=\diag(L_{v_i}(T_i))_{i=1}^{k}.
\]
The Aitken block-diagonalization formula (cf. \cite{Aitken1939, Zhang2005}) gives us
\begin{align*}
L(T)-\lambda=\begin{bmatrix}A-\lambda & B \\ B^\top & C-\lambda\end{bmatrix}=\begin{bmatrix}I & B(C-\lambda)^{-1} \\ 0 & I\end{bmatrix}\begin{bmatrix} S_{T_1,\ldots,T_k}(\lambda) & 0 \\ 0 & C-\lambda \end{bmatrix}\begin{bmatrix}I & 0 \\ (C-\lambda)^{-1}B^\top & I\end{bmatrix}.
\end{align*}
From this equation it is easy to see that (a) and (b) hold.

Clearly we have $\rk S_{T_1,\ldots,T_k}(\lambda)\geq k-1$, as the first $k-1$ columns of $S_{T_1,\ldots,T_k}(\lambda)$ are linearly independent. Hence from the dimension formula we have
\[
\dim\ker S_{T_1,\ldots,T_k}(\lambda)=k-\rk S_{T_1,\ldots,T_k}(\lambda) \leq 1.
\]
It remains to show that $x_1\neq 0$ and $x_k\neq 0$ for an eigenvector $(x_1,\ldots,x_k,y_1^\top,\ldots,y_k^\top)^\top\in \ker L(T)-\lambda$. Assume that $x_1=0$ then we obtain from the equation $S_{T_1,\ldots, T_k}(\lambda)x=0$ that $x_i=0$ for all $i=1,\ldots,k$ and hence from \eqref{yeq} we see that $y_i=0$ for all $i=1,\ldots,k$ which is a contradiction. For $x_k=0$ we can repeat the arguments from above.
\end{proof}

The Schur reduction can also be applied to weighted trees, i.e.\ when
each edge has a positive weight. It can also be applied if the attached graphs $T_i$ are arbitrary connected graphs.

We prove some basic properties of the functions $f_{T_i}$.
\begin{proposition}
\label{fprop}
Let $T$ be a tree decomposed as in Figure \ref{dergraph}.
Consider the tree $T_{i}$ and assume that $T_{i}$ is partitioned into subgraphs $T_{i,1},\ldots,T_{i,l}$ as in Figure \ref{GU} that have only $v_i$ as a joint vertex. Then $\sigma(L_{v_i}(T_i))=\bigcup_{i=1}^l\sigma(L_{v_i}(T_{i,l}))$ and the following
holds.
\begin{itemize}
\item[\rm (a)]
$f_{T_i}(\lambda)=\sum_{j=1}^lf_{T_{i,j}}(\lambda)$
for all $\lambda\notin \sigma(L_{v_i}(T_i))$.

\item[\rm (b)] $f_{T_i}^{(k)}(\lambda)> f_{T_i}^{(k)}(0)>0$ for all  $\lambda\in(0,\lambda_{\min}(L_{v_i}(T_i)))$ and $k\in\N$.

\item[\rm (c)] We have $f_{T_i}(0)=\deg_{T_i}(v_i)$ and for all $k\in\N\setminus\{0\}$
\[
\|L_{v_i}(T_i)\|^{-(k-1)} \leq \frac{f_{T_i}^{(k)}(0)}{k!(|V(T_i)|-1)}\leq \|L_{v_i}(T_i)^{-1}\|^{k-1}.
\]

In particular, for $k=1$, we have $f'_{T_i}(0)=|V(T_i)|-1$.
\item[\rm (d)] For all $\lambda\in(0,\lambda_{\min}(L_{v_i}(T_i)))$ we have
\[
\frac{\lambda(|V(T_i)|-1)}{1-\|L_{v_i}(T_i)\|^{-1}\lambda}\leq f_{T_i}(\lambda)-\deg_{T_i}(v_i)\leq \frac{\lambda(|V(T_i)|-1)}{1-\|L_{v_i}(T_i)^{-1}\|\lambda}.
\]
\item[\rm (e)]
Let $S_i$ be a subtree of $T_i$ with $v_i\in V(S_i)$ that can be obtained from removing step by step pendant vertices. Then we have $\lambda_{\min}(L_{v_i}(T_i))\leq \lambda_{\min}(L_{v_i}(S_i))$,  $f_{S_i}^{(k)}(0)\leq f_{T_i}^{(k)}(0)$ for all $k\in\N$ and
\[
f_{S_i}(\lambda)\leq f_{T_i}(\lambda),\quad  \lambda\in(0,\lambda_{\min}(L_{v_i}(T_i))).
\]
If in addition $S_i\neq T_i$, then all of the above inequalities are strict.

\end{itemize}
\end{proposition}
\begin{figure}[htbp!]
    \centering
    \includegraphics[scale=0.7]{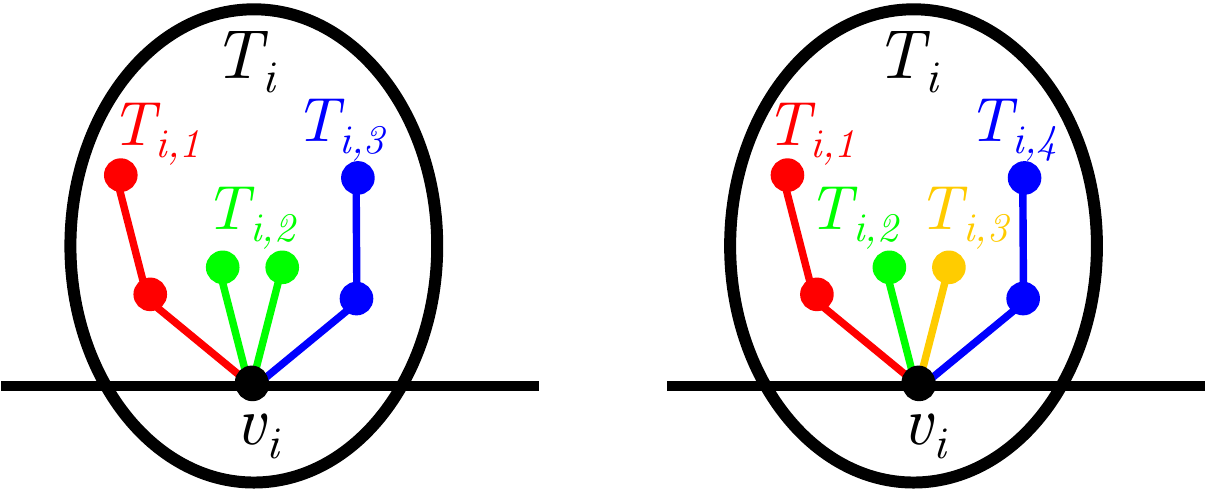}
    \caption{The figure illustrates the situation in Proposition \ref{fprop}. We see two possible partitions of the tree $T_i$ into subtrees $T_{i,j}$.}
    \label{GU}
\end{figure}

\begin{proof}
After a relabelling of vertices we have
\[
L_{v_i}(T_i)=\diag(L_{v_i}(T_{i,1}),\ldots,L_{v_i}(T_{i,l}))
\]
and therefore $\sigma(L_{v_i}(T_i))=\bigcup_{i=1}^l\sigma(L_{v_i}(T_{i,l}))$ holds. We decompose the vector
\[
f_i=(f_{i,1}^\top,\ldots,f_{i,l}^\top)^\top
\]
where $f_{i,j}\in\R^{|V(T_{i,j})|-1}$ is a vector that is zero except for one entry $-1$ corresponding to the vertex that is the unique neighbor of $v_i$ in $T_{i,j}$.
Thus, we see that
\begin{align}
f_{T_i}(\lambda)&=f_i^\top(L_{v_i}(T_i)-\lambda)^{-1}f_i\nonumber \\ &=f_i^\top\diag((L_{v_i}(T_{i,1})-\lambda)^{-1},\ldots,(L_{v_i}(T_{i,l})-\lambda)^{-1})f_i\nonumber\\ &=\sum_{j=1}^lf_{i,j}^\top(L_{v_i}(T_{i,j})-\lambda)^{-1}f_{i,j}\label{summe}\\ &=\sum_{j=1}^lf_{T_{i,j}}(\lambda)\nonumber
\end{align}
which proves (a). The function $f_{T_i}$ is analytic on $[0,\lambda_{\min}(L_{v_i}(T_i)))$ with derivatives given by
\begin{align}
\label{abltneu}
f_{T_i}^{(k)}(\lambda)=k!f_i^\top(L_{v_i}(T_i)-\lambda)^{-(k+1)}f_i,\quad k\geq 1.
\end{align}
From the choice of $\lambda\in(0,\lambda_{\min}(L_{v_i}(T_i)))$ and the Weyl bound \cite[Theorem 4.3.1]{Horn2013} implies that $L_{v_i}(T_i)-\lambda$ is positive definite and hence also $(L_{v_i}(T_i)-\lambda)^{-(k+1)}$ is positive definite for all $k\geq 1$. Thus, the right hand side in \eqref{abltneu} is positive, as $f_i\neq 0$. This implies that $f^{(k)}_{T_i}$ is strictly monotonically increasing on $[0,\lambda_{\min}(L_{v_i}(T_i)))$, which proves (b).

For the proof of the assertion (c) we use \eqref{summe} for $\lambda=0$ and $l=\deg_{T_i}(v_i)$. Let the indices $k$ and $l$ of $L_{v_i}(T_{i,j})^{-1}$ correspond to the vertices $v,w\in V(T_i)$, respectively. Since $T_i$ is a tree there are unique paths $P_{v,v_i}$ and $P_{w,v_i}$ in $T_i$ from $v_i$ to $v$ and $w$, respectively.
Then it was shown in \cite[Proposition 1]{Kirkland1996} that the entry of $L_{v_i}(T_{i,j})^{-1}$ with index $(k,l)$ equals $|E(P_{v,v_i})\cap E(P_{w,v_i})|$, i.e.\ the number of joint edges of both paths. We consider the diagonal entry that arises from $v=w=v_i^{(j)}$, where $v_i^{(j)}$ is the unique neighbor for $v_i$ in $T_{i,j}$. Then $P_{v,v_i}=P_{w,v_i}$ is a path of length one and this gives us
\[
f_{i,j}^\top L_{v_i}(T_{i,j})^{-1}f_{i,j}=1.
\]
Using this together with \eqref{summe}, we see that the first equality in (c) holds.
The characterization of the entries of $L_{v_i}(T_i)^{-1}$ from \cite[Proposition 1]{Kirkland1996} yields 
\begin{align}
\label{norm1vec}
f_i^\top L_{v_i}(T_i)^{-2}f_i=(L_{v_i}(T_i)^{-1}f_i)^\top L_{v_i}(T_i)^{-1}f_i=\|(1,\ldots,1)^\top\|^2=|V(T_i)|-1.
\end{align}
The Cauchy-Bunjakowski inequality applied to \eqref{abltneu} gives with \eqref{norm1vec}
\begin{align*}
f_{T_i}^{(k)}(\lambda)&=k!f_i^\top(L_{v_i}(T_i)-\lambda)^{-(k+1)}f_i\leq k!\|L_{v_i}^{-1}\|^{(k-1)}\|L_{v_i}(T_i)f_i\|^2 =k!\|L_{v_i}^{-1}\|^{(k-1)}(|V(T_i)|-1).
\end{align*}
This is the upper bound in (c). The proof of the lower bound in (c) is similar using  $\lambda_{\min}(L_{v_i}^{(k-1)})$ and that $\lambda_{\min}(L_{v_i}^{-1})=\|L_{v_i}\|^{-1}$ holds.

For the proof of (d) we use \eqref{abltneu} and obtain from a Taylor expansion of $f_{T_i}$ at $\lambda=0$
\begin{align*}
f_{T_i}(\lambda)&=\deg_{T_i}(v_i)+\sum_{k=2}^{\infty}\frac{f_{T_i}^{(k)}(0)\lambda^k}{k!}\\ &\leq \deg_{T_i}(v_i)+\sum_{k=1}^{\infty}(|V(T_i)|-1)\|L_{v_i}(T_i)^{-1}\|^{k-1}\lambda^k\\ &\leq \deg_{T_i}(v_i)+(|V(T_i)|-1)\lambda\sum_{k=0}^{\infty}\|L_{v_i}(T_i)^{-1}\|^{k}\lambda^k\\&=\deg_{T_i}(v_i)+\frac{(|V(T_i)|-1)\lambda}{1-\lambda\|L_{v_i}(T_i)^{-1}\|}.
\end{align*}
This proves the upper bound in (d). The lower bound can be obtained similarly, by using the lower bound for $f_{T_i}^{(k)}(0)$ from (c).

We continue with the proof of (e). Since $S_i$ can be obtained from $T_i$ by removing pendant vertices, the matrix $L_{v_i}(S_i)$ can be obtained from  $L_{v_i}(T_i)$ after applying negative rank one perturbations in combination with deletion of rows and columns with the same index. Therefore, we see from Weyl's interlacing inequality \cite[Corollary 4.3.9]{Horn2013} and Cauchy's interlacing  inequality \cite[Theorem 4.3.17]{Horn2013} that $\lambda_{\min}(L_{v_i}(T_i))\leq \lambda_{\min}(L_{v_i}(S_i))$.
From the subgraph condition and the choice of $v_i$ in $S$ we have the following inequality for the entries of the reduced Laplacians
\begin{align}
\label{entries}
0< (L_{v_i}(S_i)^{-1})_{k,l} \leq (L_{v_i}(T_i)^{-1})_{k,l}
\end{align}
for all $k,l=1,\ldots,|V(S_i)|-1$ where we assume that the entries of the matrices are sorted in such a way that the corresponding vertices of $S_i$ in $T_i$ have the same index.
From this it is easy to see that $f_{S_i}^{(k)}(0)\leq f_{T_i}^{(k)}(0)$ for all $k\in\N$. From the Taylor expansion of $f_{T_i}$ and $f_{S_i}$ at 0 we see that $f_{S_i}(\lambda)\leq f_{T_i}(\lambda)$.
Let $S_i$ be a proper subtree of $T_i$ then the matrix $L_{v_i}(S_i)$ is a proper submatrix of $L_{v_i}(T_i)$ and the strictness of the inequalities follows from the positivity of the entries in \eqref{entries}.
\end{proof}

The bound in (d) is holds with equality if $T_i$ is a star graph with center vertex $v_i$, because $T_i$ can be decomposed into $T_{i,j}$ with  $j=1,\ldots,\deg_{T_i}(v_i)$ which are graphs that consist of two vertices and  one edge between them.

Note that the value $\lambda_{\min}(L_v(T)))^{-1}$ equals $\|L_v(T)^{-1}\|$ which is known as the \textit{Perron value} in the literature, see \cite{Kirkland1996}.
In the lemma below, we provide some upper and lower bounds for $\lambda_{\min}(L_{v_i}(T_i))$, and hence, for the Perron value, see also  \cite[Theorem 4.2]{AndradeDahl2017}.
\begin{lemma}
\label{evbound}
Let $T$ be a tree and let $v$ be a vertex in $T$ and  consider for each pendant vertex $w$ the path $v_1=v\ldots v_{d(v,w)+1}=w$ and let $T_i$ be the tree associated with $v_i$ then we have
\begin{align*}
\max_{\deg_T(w)=1}\sqrt{\sum_{i=0}^{d(v,w)}i^2|V(T_{i+1})|}&\leq \lambda_{\min}(L_v(T))^{-1}=\|L_v(T)^{-1}\|\leq \max_{\deg_T(w)=1}\sum_{i=0}^{d(v,w)}i|V(T_{i+1})|.
\end{align*}
\end{lemma}
\begin{proof}
Using the spectral radius we find $\rho(L_{v}(T)^{-1})\leq \|L_{v}(T)^{-1}\|_{\infty}$ and hence
\[
\lambda_{\min}(L_v(T))=\|L_{v}(T)^{-1}\|^{-1}=\rho(L_{v}(T)^{-1})^{-1}\geq \|L_{v}(T)^{-1}\|_{\infty}^{-1}.
\]

Now the upper bound is a simple consequence of the formula for the entries of $L_{v}(T)^{-1}$ from \cite[Proposition 1]{Kirkland1996}. It is easy to see that the maximum over the row sums is attained at rows that correspond to a pendant vertex $w$.
The lower bound follows from the trivial estimate $\|L_v(T)^{-1}\|\geq\|L_v(T)^{-1}e_i\|$ where $e_i$ is a canonical unit vector and taking the maximum over those unit vectors whose index corresponds to the pendant vertices in $T$.

\end{proof}
The bounds above hold with equality for $T$ given by $V(T)=\{v,w\}$ and $E(T)=\{vw\}$.

\section{On the ratio of Laplacian eigenvector entries}
\label{sec:ratios}
In this section we use the Schur reduction in order to compare two eigenvector entries. In the following, we assume that a path $v_1\ldots v_k$ is given in $T$ with associated trees $T_i$ (see Figure \ref{dergraph}).

First we consider the case $k=2$ and $k=3$, i.e.\ we study the eigenvector entries at vertices with distance less than or equal to two. In this case the ratio of the entries can be described in terms of the functions $s_{T_1}$ and $s_{T_2}$ from \eqref{def:fTi}.
\begin{proposition}
Let $T$ be given as in Figure \ref{dergraph} with $\lambda\in \sigma(L(T))$ and associated eigenvector $x=(x_1,\ldots,x_n)^\top$.
\begin{itemize}
\item[\rm (a)]
Assume that $k=2$ and $\lambda\notin\sigma(L_{v_1}(T_1))\cup\sigma(L_{v_2}(T_2))$, then the entries $x_1$ and $x_2$ of the eigenvector $x$ for $\lambda$ at $v_1$ and $v_2$, respectively, satisfy $x_1,x_2\neq 0$ and
\begin{align}
\label{x1x2}
\frac{x_2}{x_1}=s_{T_1}(\lambda)=s_{T_2}(\lambda)^{-1}.
\end{align}
\item[\rm (b)] Assume that $k=3$ and $\lambda\notin\sigma(L_{v_i}(T_i))$ for $i=1,2,3$. Then the entries $x_1$ and $x_3$ of the eigenvector $x$ for $\lambda$ at $v_1$ and $v_3$, respectively, satisfy $x_1,x_3\neq 0$ and 
\[
\frac{x_3}{x_1}=\frac{s_{T_1}(\lambda)}{s_{T_3}(\lambda)}.
\]
\end{itemize}
\end{proposition}
\begin{proof}
According to Theorem \ref{thm:schur}, the eigenvector entries $x_1$ and $x_2$ are the solution of the equation
\begin{align}
\label{x1x2null}
\begin{pmatrix}s_{T_1}(\lambda)&-1\\-1 &s_{T_2}(\lambda)\end{pmatrix}\begin{pmatrix}x_1\\ x_2\end{pmatrix}=0.
\end{align}

Now the formula \eqref{x1x2} immediately follows from \eqref{x1x2null}.

We continue with the proof of (b). Applying the Schur reduction to the trees $T_1$, $T_2$ and $T_3$ leads to the following system of equations
\begin{align}
\label{123}
\begin{pmatrix}s_{T_2}(\lambda)&-1&-1\\-1&s_{T_1}(\lambda)&0\\-1&0&s_{T_3}(\lambda)\end{pmatrix}\begin{pmatrix}x_2\\x_1\\x_3\end{pmatrix}=0.
\end{align}
Again, we have from Theorem \ref{thm:schur} that $x_1,x_3\neq 0$ and solving the second and third component of the equation \eqref{123} for $x_1$ we see that (b) holds.
\end{proof}

In the remainder of this section we consider the case that $k\geq 3$ and we assume that $v_1$ is a pendant vertex, i.e.\ $V(T_1)=\{v_1\}$. This allows us to compare the values of the eigenvectors at pendant vertices.
We denote the subgraph that contains the path $v_1\ldots v_k$ and the trees $T_1,\ldots,T_{k-1}$ by $\hat T$.

\begin{figure}[htbp!]
    \centering
    \includegraphics[scale=0.5]{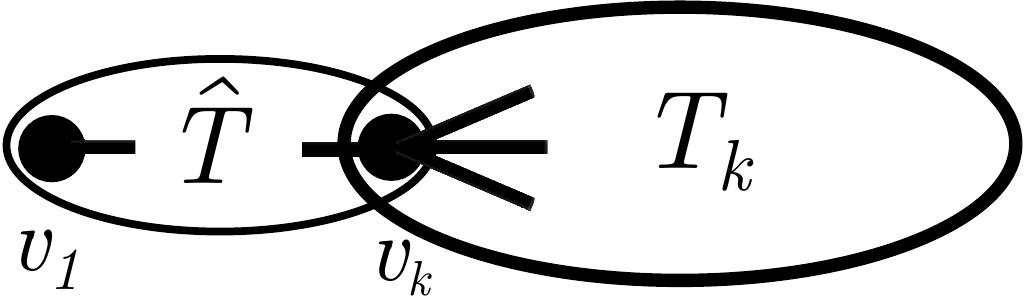}
    \caption{To compare the eigenvector entries at $v_1$ and $v_k$, we consider the subgraph $\hat T$ that contains the path $v_1\ldots v_k$ and the associated trees $T_1,\ldots,T_{k-1}$ as in Figure \ref{dergraph}.}
    \label{newred}
\end{figure}

Let $L_{v,w}(T)$ be the doubly reduced Laplacian, i.e.\ the matrix that is obtained by  deleting the row and the column corresponding to the vertices $v$ and $w$, then we can write $L(T)$ as
\[
L(T)=\begin{bmatrix} L_{v_1,v_k}(\hat T)&-e_1&-e_{k-2} &0\\ -e_1^\top& \deg(v_1) &  0&0\\-e_{k-2}^\top&0&\deg(v_k)&f_k^\top\\0&0&f_k&L_{v_k}(T_k) \end{bmatrix}
\]
where $e_1, e_{k-2}\in\R^{|V(\hat T)|-2}$ are the canonical unit vectors and $f_k$ is a vector with entries $0$ and $-1$ describing the adjacency of $v_k$ with vertices in $T_k$. Let $x_1$ and $x_k$ be the entries of the eigenvector for $\lambda\in\sigma(L(T))\setminus\sigma(L_{v_1,v_k}(\hat T))$. Then, we consider the kernel equations of the Schur complement $((L(T)-\lambda)/(L_{v_1,v_k}(\hat T)-\lambda))$ with  $\deg(v_1)=1$ leading to the equation
\begin{align*}
(1-\lambda-e_{1}^\top(L_{v_1,v_k}(\hat T)-\lambda)^{-1}e_{1})x_1-e_{1}^\top(L_{v_1,v_k}(\hat T)-\lambda)^{-1}e_{k-2}x_k=0.
\end{align*}
We introduce the function $g_{v_1,v_k}:[0,\lambda_{\min}(L_{v_1,v_k}(\hat T)))\rightarrow\R$ given by
\begin{align}
\label{quotient}
g_{v_1,v_k}(\lambda):=\frac{1-\lambda-e_{1}^\top (L_{v_1,v_k}(\hat T)-\lambda)^{-1}e_{1}}{e_{k-2}^\top (L_{v_1,v_k}(\hat T)-\lambda)^{-1}e_1},
\end{align}
then Theorem \ref{thm:schur} implies for the eigenvector entries $x_1$ and $x_k$ at $v_1$ and $v_k$, respectively,
\[
g_{v_1,v_k}(\lambda)=\frac{x_k}{x_1}.
\]

In the lemma below we state some properties of the function $g_{v_1,v_k}$.
\begin{lemma}
\label{thm:pendant}
Let $T$ be a tree decomposed into $\hat T$ and $T_k$ as in Figure \ref{newred} with $k\geq 3$ and $\deg_T(v_1)=1$.
Then $g_{v_1,v_k}$ is strictly monotonically decreasing  and
\[
g_{v_1,v_k}(0)=1,\quad g_{v_1,v_k}'(0)=1-k-\sum_{i=1}^{k-2}i|V(T_{i+1})|.
\]
\end{lemma}
\begin{proof}
For $\lambda\in\rho(L_{v_1,v_k}(\hat T))$ we introduce
\[
g_1(\lambda):=1-\lambda-e_1^\top(L_{v_1,v_k}(\hat T)-\lambda)^{-1}e_1,\quad g_2(\lambda):=e_{k-2}^\top(L_{v_1,v_k}(\hat T)-\lambda)^{-1}e_1.
\]
Since $L_{v_1,v_k}(\hat T)-\lambda$ is a matrix that has only positive diagonal entries and non-positive off-diagonal entries it follows from \cite[Theorem 4.3]{Fiedler1962} that $(L_{v_1,v_k}(\hat T)-\lambda)^{-1}$ has non-negative entries only.
Hence, the derivatives of $g_1$ and $g_2$ satisfy for all $\lambda\in[0,\lambda_{\min}(L_{v_1,v_k}(\hat T)))$
\[
g_1'(\lambda)=-1-e_1^\top(L_{v_1,v_k}(\hat T)-\lambda)^{-2}e_1<0,\quad g_2'(\lambda)=e_{k-2}^\top(L_{v_1,v_k}(\hat T)-\lambda)^{-2}e_1\geq 0.
\]
This implies that $g_1$ is strictly monotonically decreasing and that $g_2$ is strictly monotonically increasing on $[0,\lambda_{\min}(L_{v_1,v_k}(\hat T)))$.
We will show in the second part of the proof that $g_2(0)=\frac{1}{k-1}>0$ which implies
$g_2(\lambda)>0$ for all  $\lambda\in[0,\lambda_{\min}(L_{v_1,v_k}(\hat T)))$.
Therefore the function $g_{v_1,v_k}$ satisfies for all $0\leq\lambda_1<\lambda_2<\lambda_{\min}(L_{v_1,v_k}(\hat T))$
\[
g_{v_1,v_k}(\lambda_2)=\frac{g_1(\lambda_2)}{g_2(\lambda_2)}<\frac{g_1(\lambda_1)}{g_2(\lambda_2)}\leq \frac{g_1(\lambda_1)}{g_2(\lambda_1)}=g_{v_1,v_k}(\lambda_1)
\]
and is therefore strictly monotonically decreasing on  $(0,\lambda_{\min}(L_{v_1,v_k}(\hat T)))$.

It remains to compute $g_{v_1,v_k}(0)$ and $g_{v_1,v_k}'(0)$. A short computation shows that
\[
g_{v_1,v_k}'(0)=
\frac{(-1-e_{1}^\top L_{v_1,v_k}(\hat T)^{-2}e_{1})e_{k-2}^\top L_{v_1,v_k}(\hat T)^{-1}e_1-(1-e_{1}^\top L_{v_1,v_k}(\hat T)^{-1}e_{1})e_{k-2}^\top L_{v_1,v_k}(\hat T)^{-2}e_1}{(e_{k-2}^\top L_{v_1,v_k}(\hat T)^{-1}e_1)^2}. 
\]

In the following we make a construction to apply the formula for the inverse of the reduced Laplacian from \cite[Proposition 1]{Kirkland1996}.

We use that $L_{v_1,v_k}(\hat T)=L_{v_0}(\tilde{T})$ where $\tilde T$ is the graph obtained from $\hat T$ after merging $v_1$ and $v_k$ to one vertex $v_0$ with degree $2$. More precisely, we have $V(\tilde T):=(V(\hat T)\setminus\{v_1,v_k\})\cup\{v_0\}$ and $E(\tilde T):=(E(\hat T)\setminus\{v_1v_2,v_{k-1}v_k\})\cup\{v_0v_2,v_0v_{k-1}\}$.

The graph $\tilde T-v_0v_{k-1}$, where we delete the edge that connects $v_0$ with $v_{k-1}$, is a tree such that the formula from \cite[Proposition 1]{Kirkland1996} can be used. From the Sherman-Morrison-Woodbury formula, see e.g.\  \cite{Bartlett1951}, we conclude
\[
L_{v_1,v_k}(\hat T)^{-1}=L_{v_0}(\tilde T)^{-1}=L_{v_0}(\tilde T-v_0v_{k-1})^{-1}-\frac{L_{v_0}(\tilde T-v_0v_{k-1})^{-1}e_{k-2}e_{k-2}^\top L_{v_0}(\tilde T-v_0v_{k-1})^{-1}}{1+e_{k-2}^\top L_{v_0}(\tilde T-v_0v_{k-1})^{-1}e_{k-2}}
\]

and with \cite[Proposition 1]{Kirkland1996} we have
\[
1+e_{k-2}^\top L_{v_0}(\tilde T-v_0v_k)^{-1}e_{k-2}=k-1,
\]
and hence, after a suitable permutation of the entries, it holds that
\begin{align*}
e_{k-2}^\top L_{v_1,v_k}(\hat T)^{-1}&=\frac{1}{k-1}e_{k-2}^\top L_{v_0}(\tilde T-v_0v_{k-1})^{-1}\\&=\frac{1}{k-1}(\underbrace{k-2,\ldots,k-2}_{\in\R^{1\times |V(T_{k-1})|}},\ldots,\underbrace{1,\ldots,1}_{\in\R^{1\times |V(T_2)|}}),\\
L_{v_1,v_k}(\hat T)^{-1}e_{1}&=L_{v_0}(\tilde T-v_0v_{k-1})^{-1}e_1-\frac{1}{k-1}L_{v_0}(\tilde T-v_0v_{k-1})^{-1}e_{k-2}\\
&=\frac{1}{k-1}(\underbrace{1,\ldots,1}_{\in\R^{1\times |V(T_{k-1})|}},\ldots,\underbrace{k-2,\ldots,k-2}_{\in\R^{1\times |V(T_{2})|}})^\top
\end{align*}
and hence
\[
e_{k-2}L_{v_1,v_k}(\hat T)^{-1}e_1=\frac{1}{k-1},\quad e_{1}^\top L_{v_1,v_k}(\hat T)^{-1}e_1=\frac{k-2}{k-1}.
\]
This implies that for $\lambda=0$ in \eqref{quotient} we have $g_{v_1,v_k}(0)=1$.

Furthermore, we see that
\begin{align*}
g'_{v_1,v_k}(0)&=(k-1)(-1-e_1^\top L_{v_1,v_k}(\hat T)^{-2}e_{k-2}-e_1^\top L_{v_1,v_k}(\hat T)^{-2}e_1)\\
&=(k-1)(-1-e_1^\top L_{v_1,v_k}(\hat T)^{-1}(L_{v_1,v_k}(\hat T)^{-1}e_{k-2}+L_{v_1,v_k}(\hat T)^{-1}e_1)\\
&=(k-1)(-1-e_1^\top L_{v_1,v_k}(\hat T)^{-1}(1,\ldots,1)^\top)\\
&=1-k-\sum_{i=1}^{k-2}i|V(T_{i+1})|.
\end{align*}
\end{proof}

The corollary below is a consequence of Theorem \ref{thm:schur} and Lemma \ref{thm:pendant}. Here we  compare the eigenvector entries at two pendant vertices.
\begin{corollary}
\label{cor:pendant}
Let $T$ be a tree with pendant vertices $w$ and $w'$ and consider a vertex $v$ with $w'v,wv\notin E(T)$ and let $\hat T$ and $\hat T'$ be the trees from Figure \ref{newred} for $w=v_1$ and $w'=v_1$ and $v=v_k$.
Assume that $\lambda\in\sigma(L(T))$ with $\lambda <\min\{\lambda_{\min}(L_{w,v}(\hat T)),\lambda_{\min}(L_{w',v}(\hat T'))\}$ then the entries $x_w$ and $x_{w'}$ of the associated eigenvector satisfy
\[
g_{w,v}(\lambda)x_w=g_{w',v}(\lambda)x_{w'}.
\]
Let $\lambda$ be sufficiently small, then $x_v,x_w,x_{w'}\neq0$ and $\frac{x_{w'}}{x_w}\geq 1$ if there exists an index $i_0\geq 0$ with $g_{w,v}^{(i)}(0)=g_{w',v}^{(i)}(0)$ for all $i=0,\ldots,i_0$ and $g_{w,v}^{(i_0+1)}(0)>g_{w',v}^{(i_0+1)}(0)$.
\end{corollary}
We will see later that in the special case $\lambda=a(T)$, a large diameter of $T_k$ implies that the value $a(T)$ is small (cf.\ \eqref{atbounds}).

\section{Extremal entries of the Fiedler vector}
\label{sec:fiedler}

In this subsection, we consider the Fiedler vector, which is the eigenvector corresponding to the first nonzero eigenvalue $a(T)$ of $L(T)$. Here it is assumed that $a(T)$ is a simple eigenvalue of $L(T)$.
Since $T$ is connected, the vector $(1,\ldots,1)^\top$ is the, up to scaling, unique eigenvector for the eigenvalue $0$. Since $L(T)$ is a symmetric matrix, the Fiedler vector is orthogonal to $(1,\ldots,1)^\top$ and hence, it contains both, positive and negative entries.

In the following, we will study the extremal entries. These entries were also studied in \cite{Lefevre2013}, where it was said that a graph has the \textit{Fiedler extrema diameter} (FED) property, if the Fiedler vector has only two extrema that are located at the endpoints of the longest path.

In the lemma below, we show that the extremal entries of the Fiedler vector are located only at the pendant vertices, see also \cite[Corollary 1]{Lefevre2013}.
\begin{lemma}
\label{fiedlercor}
Let $T$ be a tree with a path $v_1\ldots v_k$ as in Figure \ref{dergraph} with diameter $d(T)\geq 2$ and Fiedler vector $x=(x_i)_{i=1}^n$, then the following holds.
\begin{itemize}
\item[\rm (a)] The extremal entries of the Fiedler vector are located only at pendant vertices.

\item[\rm (b)] One of the following assertions holds:
\begin{itemize}
\item[\rm (i)] The entries of the Fiedler vector on the path $v_1\ldots v_k$ are monotonically decreasing or increasing.
\item[\rm (ii)] We have $x_1,x_k\geq 0$ and there exists  $x_{k'}\geq 0$ with $x_1\geq x_2\ldots\geq x_{k'}$ and $x_{k'}\leq x_{k'+1}\leq\ldots\leq x_k$.
\item[\rm (iii)] We have $x_1,x_k\leq 0$ and there exists  $x_{k'}\leq 0$ with $x_1\leq x_2\ldots\leq x_{k'}$ and $x_{k'}\geq x_{k'+1}\geq\ldots\geq x_k$.
\end{itemize}
\end{itemize}
\end{lemma}
\begin{proof}
Let $(1,\ldots,1)^\top\in\R^n$ be the eigenvector corresponding to $0$ and let $x=(x_1,\ldots,x_n)^\top$ be the up to scaling unique eigenvector corresponding to $a(T)$. From \cite[Corollary 2.3]{Fiedler1975} we have that for any $\alpha\geq 0$ the induced subgraph  $T_{\alpha}$ with vertices given by $V(T_{\alpha})=\{v_i: x_i+\alpha\geq 0\}$ is connected. This implies that there is no negative local minimum of the vector on the entries $x_2,\ldots,x_{n-1}$, i.e.\ $x_k\leq x_{k-1}$ and $x_k\leq x_{k+1}$ cannot hold for all  $k=2,\ldots,n-1$. Repeating the arguments above with $-x$, we see that there is no positive local maximum of $x$ on $v_2,\ldots,v_{n-1}$.

By choosing $x_1$ and $x_k$ as pendant vertices, we see that the extremal values are attained at the pendant vertices. It remains to show that they are only attained at pendant vertices. Since $d(T)\geq 2$ we have from
\cite[p.\ 187]{Cvetkovic1980} that
\[
a(T)\leq 2\left(1-\cos\left(\frac{\pi}{d(T)+1}\right)\right)<1.
\]
Let $v_1$ be a pendant vertex with neighbor $v_2$ then the eigenvector equation $(1-a(T))x_1=x_2$. Since $a(T)<1$, we have $x_1>x_2>0$ or $x_1<x_2<0$. Hence only the values at the pendant vertices are extremal. This finishes the proof of (i). Now (ii) and (iii) are a simple consequence of the previous arguments on the non-existence of local maxima and minima.

\end{proof}

As a first application, we consider \textit{caterpillar trees} which are trees that consist of one central path to which all other vertices have distance one. These graphs have been well investigated and find applications in chemistry and physics \cite{Harary1973,El-Basil1987}. Here the trees $T_i$ in Figure \ref{dergraph} are star graphs $S_{r_i}$ with $r_i\in\N$ and center $v_i$ and we can further decompose $T_i$ into the trees $T_{i,j}$ for $j=1,\ldots,\deg_{T_i}(v_i)$ which consist of one edge only. In this case $L_{v_i}(T_{i,j})=1$ and hence  $\sigma(L_{v_i}(T_i))=\{1\}$.

\begin{corollary}
\label{cor:cater}
Let $T$ be a caterpillar tree with $d(T)\geq 3$. Then, the extremal values of the Fiedler vector are attained at the pendant vertices in $T_1$ and in $T_k$. 
\end{corollary}
\begin{proof}
This is essentially a consequence of Lemma \ref{fiedlercor}. But we have to exclude that $x_1=x_2$ holds, i.e.\ we have strict monotonicity $x_1>x_2>0$ or $x_1<x_2<0$. To see this we use the Schur reduction with $\lambda=a(T)<1=\lambda_{\min}(L_{v_1}(T_1))$. Considering the kernel of $S_{T_1,\ldots,T_k}(a(T))$, we see from \eqref{x1x2} that
\[
\frac{x_2}{x_1}=s_{T_1}(\lambda)=1+\deg_{T_1}(v_1)-\lambda-\frac{\deg_{T_1}(v_1)}{1-\lambda}.
\]
The assumption $x_1=x_2$ and $\lambda=a(T)$ leads to
\[
1=1-a(T)-\frac{a(T)\deg_{T_1}(v_1)}{1-a(T)}
\]
and therefore $a(T)=0$ or $a(T)=\deg_{T_1}(v_1)+1>1$ which is not possible. Thus, we have shown that $x_1\neq x_2$. A similar argument shows that $x_{k-1}\neq x_k$ holds and therefore the extremal entries are located at the pendant vertices in $T_1$ and $T_k$.
\end{proof}

From Corollary \ref{cor:cater} we see that the (FED) property only holds if we assume that the caterpillar tree also satisfies $|V(T_1)|,|V(T_k)|= 2$.

The theorem below provides properties of the entries of the Fiedler vector for graphs that, are slightly more general than caterpillar trees.
\begin{theorem}
\label{thm:gencater}
Let $T$ be a tree which is decomposed into subtrees $T_{i,j}$ with $i=1,\ldots,k$ and  $j=1,\ldots,\deg_{T_i}(v_i)$ then the following holds.
\begin{itemize}
\item[\rm (a)] If $a(T)< \lambda_{\min}(L_{v_i}(T_{i,j}))$ for all $i=1,\ldots,k $, $j=1,\ldots,\deg_{T_i}(v_i)$ then $a(T)$ is a simple eigenvalue of $L(T)$.
\item[\rm (b)] Under the assumption of (a) assume additionally that $k\geq 3$ and that there exist pendant vertices $w_j\in V(T_j)$ with $j=1,k$ such that
    \begin{align*}
   0<g_{w_j,v_j}(a(T))&\leq  \min_{\scriptsize\begin{matrix}i\in\{1,\ldots,k\}\setminus\{j\}\\ v\in V(T_i),\deg_{T_i}(v)=1 \end{matrix}}g_{v,v_i}(a(T)),\quad j=1,k.
    \end{align*}
Then the values of the Fiedler vector at $w_1$ and $w_k$ have a different sign and are extremal. If these vertices are unique, then the property (FED) holds.
\item[\rm (c)] Assume that there are pendant vertices $w,w'\in V(T_i)$ from some $i=1,\ldots,k$ with disjoint paths $w=v_1,\ldots,v_{d(w,v)-1}=v$ and $w'=v_1',\ldots,v_{d(w',v)-1}'=v$ and such that
    \[
    -d(w,v)-\sum_{i=1}^{d(w,v)-1}i|V(S_{i+1})|>    -d(w',v)-\sum_{i=1}^{d(w',v)-1}i|V(S_{i+1}')|
    \]
    holds,
where $S_i$ and $S_i'$ are the trees associated with the path $w=v_1,\ldots,w_{d(v_1,v)-1}=v$ and $w'=v_1',\ldots,w_{d(v_1,v)-1}'=v$, respectively.
Then for $k$ sufficiently large, the entries $x,x'$ of the Fiedler vector in $T_i$ at $w,w'$ satisfy $\frac{x'}{x}\geq 1$.
\end{itemize}

\end{theorem}
\begin{proof}
Given that (A1) holds, then Theorem \ref{thm:schur} implies that $a(T)$ is simple. The proof of (a) is similar to the proof of Proposition \ref{Scater} and therefore omitted.

We continue with the proof of (b). Note that the assumption (A1) implies with Cauchy's interlacing inequality \cite[Theorem 4.3.17]{Horn2013}  that $a(T)\leq\lambda_{\min}(L_{v_i,v}(T_{i,j}))$ for all $v\in V(T_{i,j})$.

We apply Lemma \ref{fiedlercor} (b). From the assumption in (a) and Theorem \ref{thm:schur}, we see that $x_1\neq0$. Without restriction, we assume that $x_1>0$ holds. This excludes case (iii) in Lemma \ref{fiedlercor}.
Assume further, that we are in case (ii) of this lemma. Then all entries of the Fiedler vector $x_1,\ldots,x_k$ are nonnegative since we assumed $g_{v,v_i}(a(T))>0$ for all $i=1,\ldots,k$ and all pendant vertices $v\in V(T_i)$. Therefore, again by Lemma \ref{fiedlercor} applied to $T_i$ all vertices in $T_i$ have nonnegative values of the Fiedler vector. This implies that all entries of the Fiedler vector are nonnegative, which is not possible, since $(1,\ldots^,1)^\top\in\R^n$ is orthogonal to $x$. Therefore case (ii) in Lemma \ref{fiedlercor} (b) cannot hold.
As a consequence, we are in the case (i) in this lemma. Now we consider a path from the pendant vertex $w_1$ in $T_1$ to the pendant vertex $w_k$ in $T_k$ which must be either monotonic decreasing or increasing.
The previous arguments imply that $x_k<0$ and therefore the entries of the Fiedler vector on the selected path are decreasing. Hence the entries $w_1$ and $w_k$ have a different sign.
The assumption on $w_1$ implies that the value of the Fiedler vector at $w_1$ is extremal among all vertices in $T_1$ by Corollary \ref{cor:pendant}. It remains to show that $w_1$ is then the remaining vertices in the trees $T_2,\ldots,T_k$.
Now one can show that $x_1>x_2$, Since $d(T)\geq k\geq 3$ this can be shown in the same way as in the proof of Corollary \ref{cor:cater} which proves the previous claim. Hence $w_1$ is maximal, since we assumed that $x_1>0$ and a repetition of the arguments from above proves that the value of the Fiedler vector at $w_k$ is minimal.

For the proof of (c) we use the bound \cite[p.\ 187]{Cvetkovic1980}
\begin{equation}\label{EstSpecGap}
a(T)\leq 2-2\cos\left(\frac{\pi}{d(T)+1}\right)\leq 2-2\cos\left(\frac{\pi}{k}\right)
\end{equation}
and hence $a(T)\rightarrow0$ as $k\rightarrow\infty$. Therefore, for sufficiently large $k$, we see from Lemma \ref{thm:pendant} that the assumptions of Corollary \ref{cor:pendant} are fulfilled and thus, the assertion (c) follows.
\end{proof}

As a special case, we assume that $T_{i,j}=S$ holds for all $i=1,\ldots,k$ and all $j=1,\ldots,\deg_{T_i}(v_i)$ and some graph $S$.
In the graph $S$ we select a pendant vertex $v_0\in V(S)$ that is identified with $v_i$.
\begin{corollary}
\label{Scater}
Let $T$ be an $S$-caterpillar tree with a central path $v_1\ldots v_k$ and assume that $a(T)<\lambda_{\min}(L_{v_0}(S))$. Then there exists a Fiedler vector and the following holds:
\begin{itemize}
\item[\rm (a)] The extremal entries of the Fiedler vector are located at the pendant vertices in the trees attached to $v_1$ and to $v_k$.

\item[\rm (b)]
Assume that there is a unique vertex $v$ in $S$ that minimizes
\begin{align}
\label{maxcond}
g_{v,v_1}'(0)=-d(v_1,v)-\sum_{i=1}^{d(v_1,v)-1}i|V(S_{i+1})|
\end{align}
where $S_i$ are the unique trees associated with the path $w_1=v_1,\ldots,w_{d(v_1,v)-1}=v$.
Then the extremal entries of the Fiedler vector are attained at $v$ in the trees attached to $v_1$ and $v_k$ if $k$ is sufficiently large.
\end{itemize}
\end{corollary}
\begin{proof}
First, we need to verify that the assumptions of Theorem \ref{thm:gencater} (b) are fulfilled.
The assumption on $a(T)$ implies by Theorem \ref{thm:schur} that $x_1\neq 0$. Without restriction, we can assume that $x_1>0$. We show that $g_{v,v_1}(a(T))>0$ holds for all pendant vertices $v$ in $S$. Consider a path from $v\in V(T_1)$ to $v\in V(T_k)$ that contains $v_1$. Then we apply Lemma \ref{fiedlercor} (b) to this selected path. Since we assumed $x_1>0$ the case (iii) is not possible. Assume that (ii) holds, then obviously $g_{v,v_1}(a(T))>0$, as $x_1>0$ and the value at $v\in V(T_1)$ is positive. Assume that (i) holds and that the entries are monotonically decreasing, then the value at $v\in V(T_1)$ is greater or equal to $x_1>0$ and again $g_{v,v_1}(a(T))>0$ follows. Consider now the case that the entries are monotonically increasing then $x_k\geq x_1>0$ and also the value at $v\in V(T_k)$ is positive. Therefore $g_{v,v_1}(a(T))=g_{v,v_k}(a(T))>0$. shows that the assumptions of Theorem \ref{thm:gencater} (b) are satisfied.

\end{proof}

\begin{remark}
Assume in Corollary \ref{Scater} (b), that two pendant vertices $v,w$ both minimize \eqref{maxcond} and assume that $d(v,w)=2$, then the eigenvector entries at $v$ and $w$ are the same and hence the extremal entries are located at these vertices.
Otherwise, one has to compare higher order derivatives of $g_{v,v_1}$, to decide on which pendant vertices the extremal entry of the eigenvector is located.
\end{remark}

Finally, we discuss the example of the Fiedler rose from  \cite{Evans2011,Lefevre2013}, where paths $P_l$, $P_t$ and a star graph $S_r$ are glued together at a pendant vertex (see Figure \ref{fig:FiedlerRose}). 
\begin{figure}
    \centering
    \includegraphics[scale=0.6]{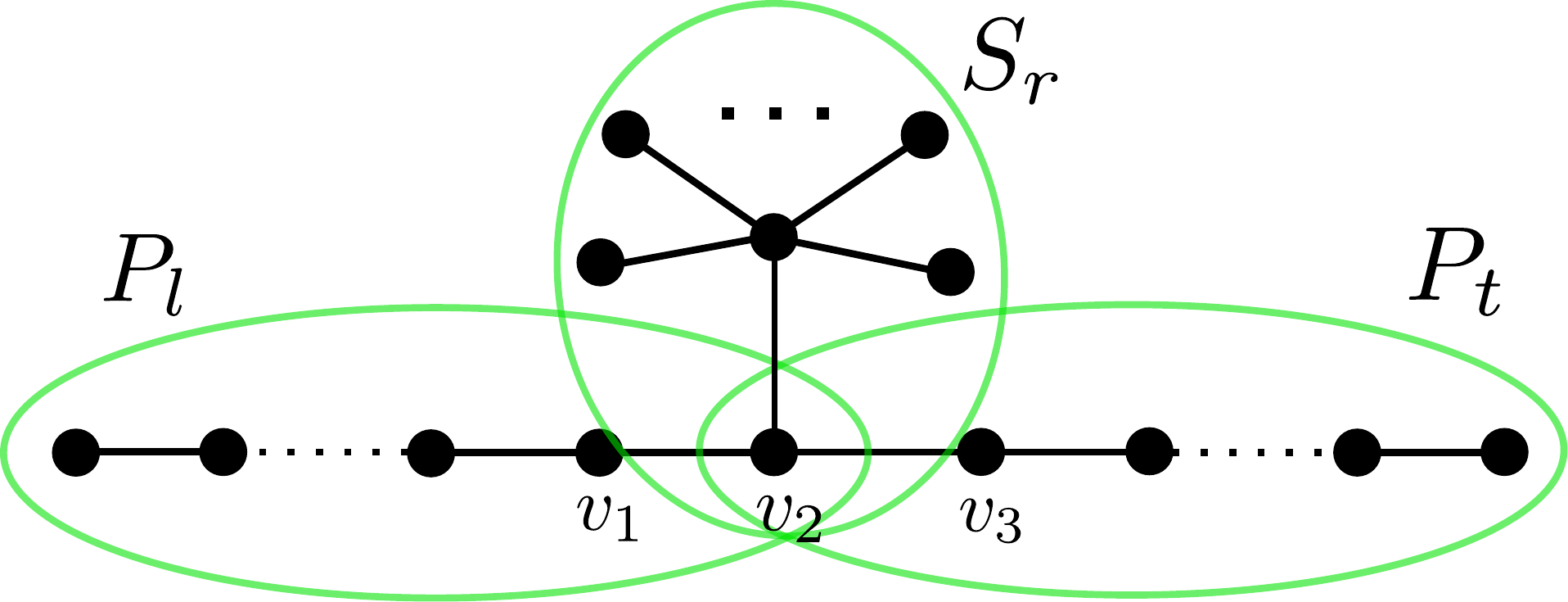}
    \caption{The Fiedler rose consisting of the paths $P_l$, $P_t$ and a star graph $S_r$ glued together at the vertex $v_2$.
    }
    \label{fig:FiedlerRose}
\end{figure}

First, we represent the rose tree as in Figure \ref{dergraph}. Without restriction, we can assume that $t\geq l$. Here we set $T_1:=P_{l-1}$,  $T_2:=S_r$ and $V(T_i):=\{v_i\}$ for $3\leq i\leq 2+t-l$ and $T_{3+t-l}:=P_{l-1}$. It is easy to see that
\[
\lambda_{\min}(L_{v_1}(T_1))=2-2\cos\left(\frac{\pi}{2l-3}\right)
\]
and hence
\[
a(T)\leq2-2\cos\left(\frac{\pi}{t+l-1}\right)<2-2\cos\left(\frac{\pi}{2l-3}\right)\leq \lambda_{\min}(L_{v_1}(T_1)).
\]
For fixed $r\geq 2$ we can always choose $l$ large enough such that
\[
a(T)< \lambda_{\min}(L_{v_2}(S_r))
\]
and that the assumptions of Theorem \ref{thm:gencater} (b) are fulfilled. More precisely, we know from the upper bound in Lemma \ref{evbound} that 
\[
\lambda_{\min}(L_{v_2}(S_r))\geq \frac{1}{r}.
\]
Together with the bound \eqref{EstSpecGap} for $a(T)$ we see that
\begin{align}
\label{alsosuff}
1-\frac{1}{2r}\leq\cos\left(\frac{\pi}{l+t-1}\right)
\end{align}
and this implies that the assumptions of Corollary \ref{cor:pendant} are fulfilled.
Furthermore, let $w_1$ be a pendant vertex in $T_1$ then we will see later, from Lemma \ref{lem:rec} that
\begin{align*}
g_{w_{1},v_2}(\lambda)=\frac{ 2}{\sqrt{4-\lambda}}\cos\left((l-1/2)\arccos\left(1-\frac{\lambda}{2}\right)\right)
\end{align*}
and for a pendant vertex $w_2$ in $T_2$ one can find from \eqref{quotient} that
\[
g_{w_{2},v_2}(\lambda)=\lambda^2-r\lambda+1.
\]
Assume now that $t>l$, then 
\[
g_{w_1,v_2}(a(T))\geq g_{w_1,v_2}\left(2-2\cos\left(\frac{\pi}{l+t-1}\right)\right)>0
\] 
and from Theorem \ref{thm:gencater} (b) we see that 
\begin{align}
\label{sufffirst}
\frac{2}{\sqrt{4-a(T)}}\cos\left((l-3/2)\arccos\left(1-\frac{a(T)}{2}\right)\right)<a(T)^2-ra(T)+1
\end{align}
is a sufficient condition for that the extremal values of the Fiedler vector are located at the endpoints of the longest path in the Fiedler rose and (FED).
Using the monotonicity of the functions $g_{w_{1},v_1}$ and $g_{w_{2},v_2}$ from Lemma \ref{thm:pendant} and the bounds \eqref{EstSpecGap} for $a(T)$ we see that the condition 
\begin{align}
\label{suffsnd}
\frac{\sqrt{2}\cos\left((l-3/2)\frac{\pi}{l+t-1}\right)}{\sqrt{1+\cos\left(\frac{\pi}{l+t-1}\right)}}<\left(2-2\cos\left(\frac{\pi}{l+t+r-1}\right)\right)^2-r\left(2-2\cos\left(\frac{\pi}{l+t+r-1}\right)\right)+1
\end{align}
is sufficient for \eqref{sufffirst} and this condition only depends on the parameters $l,t$ and $r$.
In summary, we have seen that for $t>l$ the conditions \eqref{alsosuff} and \eqref{suffsnd} are sufficient for (FED) to hold.


If $T$ is a perfect rose tree, i.e. $t=l$ then we can conclude some more structural properties of the Fiedler vector and $a(T)$. We assume again that \eqref{alsosuff} holds and let $l$ be so large that 
\begin{align}
\label{fiedlerg}
g_{w_2,v_2}(a(T))>g_{w_2,v_2}\left(2-2\cos\left(\frac{\pi}{2l-1}\right)\right)>0.
\end{align}

Assume that $a(T)<2-2\cos(\frac{\pi}{2l-1})$, then $g_{w_1,v_2}(a(T))>0$ and hence $x_{w_1},x_{v_2}\neq 0$, say $x_{w_1},x_{v_2}>0$. This implies that all entries of the eigenvector on $P_l$ are positive. Due to symmetry, also the Fiedler vector is also positive on the path $P_t$. Since $g_{w_2,v_2}(a(T))>0$ the Fiedler vector is nonnegative on $S_r$. This is a contradiction to the orthogonality to the eigenvector $(1,\ldots,1)^\top$.

Therefore, we have $a(T)=2-2\cos(\frac{\pi}{2l-1})$ which implies by Lemma \ref{lem:rec} that $x_{v_2}=0$. Hence,  $x_{v_2}=g_{w_{2},v_2}(a(T))x_{w_2}=0$ and \eqref{fiedlerg} imply that $x_{w_2}=0$. Lemma  \ref{fiedlercor} (b) implies that the Fiedler vector is zero at all vertices of $S_r$.
Furthermore, the extremal entries are located at $v_1$ and $v_{2l-1}$.

Let us now consider a rose tree with fixed $l$, but with sufficiently large $t>l$. In this setting, we assume that $T_1$ is decomposed into $T_{1,1}:=P_l$, $T_{1,2}:=S_r$ and set $T_{2}:=P_{t}$. We consider the rose tree for sufficiently large $t$, then we know from Corollary \ref{cor:pendant}, where the extremal entries are located by looking at the derivatives
\[
g'_{w_1,v_1}(0)=-\frac{l(l+1)}{2},\quad g'_{w_2,v_1}(0)=-r.
\]
If $r>\frac{l(l+1)}{2}$, then we have from Theorem \ref{thm:gencater} (c) that for sufficiently large $t$, the extremal entries of the Fiedler vector are located at $V(S_r)\setminus\{v_1\}$ and at the end of $P_l$.

On the other hand, if $r<\frac{l(l+1)}{2}$
then for sufficiently large $t$, the extremal values of the Fiedler vector are located at $v$ and the end of the path $P_l$.

The Fiedler rose is a counter example to the conjecture that the extremal values of the Fiedler vector on trees are located at the end points of the longest path. In the following we apply this construction to arbitrary trees, to show that we can add a large star graph to a vertex on the longest path such that the extremal entries are located on this star graph. In particular, it is possible to move the extremal entries from the longest path to the pendant vertices of the star graph. The corollary below follows from Theorem \ref{thm:gencater} (c).

\begin{corollary}
Let $T$ be a tree with a path $v_1\ldots v_k$ and associated trees $T_1,\ldots, T_k$ and $k=d(T)+1$.
Then $d(T_i)\leq 2(i-1)$ for all $i\leq\lceil\frac{d(T)+1}{2}\rceil$. Assume that $k\geq 7$ and that $r\in\N$ satisfies
\[
-(r+1)< -3-|V(T_2)|-2|V(T_3)|.
\]
Let $a(T)$ be sufficiently small with Fiedler vector $(x_i)_{i=1}^n$ with $x_1>0$. Then one can add $S_r$ to $T_4$, i.e.\ a pendant vertex of $S_r$ is identified with $v_4$. Let $v_S\in V(S_r)$ be a pendant vertex of the resulting graph then the entries of the Fiedler vector satisfy $\frac{x_{v_S}}{x_{v_1}}\geq 1$.
\end{corollary}

\section{Bounds on the ratio of eigenvector entries along paths}\label{sec:evalue-bounds}

In this section, we provide bounds on the ratios of Laplacian eigenvector entries that depend only on the eigenvalue, but not the resolvent of a reduced Laplacian.

The bounds are based on estimates for the entries of the  kernel elements of the tridiagonal matrix  $S_{T_1,\ldots,T_k}(\lambda)$. Here we view this matrix as a perturbation of a tridiagonal Toeplitz matrix, where we allow only perturbations on the main diagonal.

Note that classical perturbation results for eigenvectors, like the Davis-Kahan theorem (cf.\ \cite{Davis1970}) are not useful since the distance between the eigenvalues is small and also they do not provide good bounds for fixed entries.

First, we present the main result of this section. Assume that there exist $\underline\zeta_k,\bar\zeta_k\in(0,\frac{\pi}{2})$ such that
\begin{align}
\label{annlap}
\begin{split}
2(1-\cos(\bar\zeta_k))&=\lambda+\min_{i=1,\ldots,k}\left(f_{T_i}(\lambda)-\deg_{T}(v_i)\right),\\
2(1-\cos(\underline\zeta_k))&=\lambda+\max_{i=1,\ldots,k}\left(f_{T_i}(\lambda)-\deg_{T}(v_i)\right).
\end{split}
\end{align}

\begin{theorem}
\label{fiedlerbounds}
Let $\lambda\notin \bigcup\limits_{i=1}^{k-1}\sigma(L_{v_i}(T_i))$ be an eigenvalue of $L(T)$ such that \eqref{annlap} holds and let $x_1>0$.
Then we have  for all $i$ with $0 < i <\min\{\frac{\pi}{2\underline\zeta_{k-1}}+1/2,k\}$ that $x_i>0$,

\begin{align}
\label{main1}
\frac{\cos\left((i+\frac{1}{2})\underline\zeta_{k-1}\right)}{\cos\left((i-\frac{1}{2})\underline\zeta_{k-1}\right)}
\leq \frac{x_{i+1}}{x_i}
\leq \frac{\cos\left((i+\frac{1}{2})\bar\zeta_{k-1}\right)}{\cos\left((i-\frac{1}{2})\bar\zeta_{k-1}\right)}.
\end{align}
and
\begin{align}
\label{main2}
x_1\frac{\cos\left((i+\frac{1}{2})\underline\zeta_{k-1}\right)}{\cos\left(\frac{1}{2}\underline\zeta_{k-1}\right)}
\leq x_{i+1}
\leq x_1
\frac{\cos\left((i+\frac{1}{2})\bar\zeta_{k-1}\right)}{\cos\left(\frac{1}{2}\bar\zeta_{k-1}\right)}.
\end{align}

\end{theorem}

Theorem \ref{fiedlerbounds} is a direct consequence of Theorem \ref{mainthmjac} below, where we bound the entries of kernel elements of the tridiagonal matrices of the form
\begin{align}
\label{jac}
A=\begin{pmatrix}1-\varepsilon_1& -1& &\\ -1& 2-\varepsilon_2& \ddots&\\ &\ddots &\ddots&-1\\ &&-1&1-\varepsilon_n\end{pmatrix}\in\R^{n\times n}
\end{align}
with $\varepsilon_1,\ldots,\varepsilon_n\in[0,\infty)$. These matrices have the same structure as $S_{T_1,\ldots,T_k}(\lambda)$ and hence the entries of the kernel elements of \eqref{jac} are the entries of the eigenvectors of $L(T)$. If we delete the last column of $A$ it is easy to see by induction that $\rk A\geq n-1$ and hence that $\dim\ker A\leq 1$.

The matrix $A$ can be viewed as a perturbation of the graph Laplacian $L(P_n)$ of the path with $n$ vertices.

First, we investigate the case $\varepsilon_i=\varepsilon$ for some $\varepsilon\in[0,\infty)$.
For $n\geq 3$ we consider the equation $Ax=0$. Leaving out the $n$th component of this linear system of equations, we obtain
\begin{align}
x_1-x_2=\varepsilon x_1,\quad -x_{i-1}+2x_{i}-x_{i+1}=\varepsilon x_i,\quad 2\leq i\leq n-1. \label{linsys}
\end{align}
Since we erased one equation, there is always a one dimensional solution space to the system \eqref{linsys} and an explicit solution is given in the lemma below, see also \cite[Lemma 2]{Lefevre2013} and \cite[Remark 3.1]{Nakatsukasa2013}.
\begin{lemma}
\label{lem:rec}
Let $\varepsilon=2(1-\cos(\zeta))$ with $\zeta\in[0,\frac{\pi}{2})$. For fixed  $x_1\in\R$ the system \eqref{linsys} has the unique solution
\begin{align}
\label{xjformula}
x_j=\frac{\sqrt{2}\ x_1}{\sqrt{1+\cos(\zeta)}}\cos((j-1/2)\zeta),\quad j=2,\ldots,n.
\end{align}
In particular, if $x_1>0$ then $x_1>x_2>\ldots>x_{k}>0$ for all  $k<\frac{\pi}{2\zeta}+\frac{1}{2}$.
\end{lemma}
\begin{proof}
Expression \eqref{xjformula} can easily be checked by plugging it into the equations \eqref{linsys}. It remains to apply the addition theorem for cosine function. Assume now that $x_1>0$ then according to \eqref{xjformula} we have $x_j>0$ as long as
$(j-1/2)\zeta<\frac\pi 2$. Solving this for $j$ proves the proposition.
\end{proof}

Now, we present the bounds for the entries of kernel elements of \eqref{jac}.
We assume that there exist  $\underline\zeta_k,\bar\zeta_k\in(0,\frac{\pi}{2})$ satisfying
\begin{align}
\begin{split}
\label{asslambda}
\underline{\varepsilon}_{k}&:=2(1-\cos(\underline\zeta_k))\geq \max_{i=1,\ldots,k}\varepsilon_i,\\
\overline{\varepsilon}_{k}&:=2(1-\cos(\overline\zeta_k))\leq\min_{i=1,\ldots,k}\varepsilon_i.
\end{split}
\end{align}

\begin{theorem}
\label{mainthmjac}
Let $A$ be given by \eqref{jac} with $A(x_1,\ldots,x_k)^\top=0$ and $x_1>0$ such that \eqref{asslambda} holds for  $\underline\zeta_{k-1},\bar\zeta_{k-1}\in(0,\frac{\pi}{2})$.
Then we have $x_1>x_2>\ldots>x_i>0$ for all $0<i<\min\{\frac{\pi}{2\underline\zeta_{k-1}}+1/2,k\}$ and
\begin{align*}
\frac{\cos((i+1/2)\underline\zeta_{k-1})}{\cos((i-1/2)\underline\zeta_{k-1})}
\leq \frac{x_{i+1}}{x_{i}}
\leq \frac{\cos((i+1/2)\bar\zeta_{k-1})}{\cos((i-1/2)\bar\zeta_{k-1})}.
\end{align*}
\end{theorem}
\begin{proof}

Define the sequence
\begin{align*}
    F_1&:=1-\varepsilon_1,\\
    F_i&:=2-\varepsilon_i-F_{i-1}^{-1},\quad i\geq 2
\end{align*}
where the element $F_i$ is only defined if $F_{i-1}\neq 0$.
Then from the kernel equation for the matrix \eqref{jac} it is easy to see that $F_1=\frac{x_2}{x_1}$ and if $F_{i-1}=\frac{x_i}{x_{i-1}}\neq 0$ and it holds that
\[
F_i=\frac{x_{i+1}}{x_i}.
\]

Consider now an auxiliary sequence $\underline F_i$ given by $\underline  F_1:=1-\underline\varepsilon_{k-1}$ and for $i\geq 2$ by
\[
\underline F_i:=2-\underline\varepsilon_{k-1}-\underline F_{i-1}^{-1}
\]
where again $\underline F_i$ is only defined if $\underline F_{i-1}\neq 0$.
Then we have
\[
F_1=1-\varepsilon_1\geq 1-\underline\varepsilon_{k-1}=\underline F_1
\]
and it is easy to see by induction that for all $i\geq 2$ with $F_{i-1},\underline F_{i-1}\neq 0$
\begin{align}
\label{FhatF}
F_i=2-\varepsilon_i-F_{i-1}^{-1}\geq 2-\underline\varepsilon_{k-1}-\underline F_{i-1}^{-1}=\underline F_i.
\end{align}
In the following we choose $\hat x_i$ such that $\hat x_1=x_1$, $\hat x_2=(1-\mu)\hat x_1$ for some $\mu\in\mathbb{R}$ and
\begin{align}
\label{bruch}
\underline F_i=\frac{\hat x_{i+1}}{\hat x_i}.
\end{align}
Plugging this into the definition of $\underline F_i$ we obtain the a linear system of equations of the form \eqref{linsys}
\begin{align}
    \label{hilfsrec}
\hat{x}_{i+1}=(2-\underline\varepsilon_{k-1})\hat{x}_i-\hat x_{i-1},\ i\geq 2,\quad  \hat x_2=(1-\underline\varepsilon_{k-1})\hat x_1.
\end{align}

Now one can apply Lemma~\ref{lem:rec} with $\varepsilon=\underline\varepsilon_{k-1}=2(1-\cos(\underline\zeta_k))$, to see that $\hat{x}_i>0$ for the recursion given by \eqref{hilfsrec} for all $i< \min\left\{\frac{\pi}{2\underline\zeta_{k-1}}+1/2,k\right\}$. Then the right hand side in \eqref{bruch} is greater than zero and therefore by \eqref{FhatF}, also $F_i>0$ and hence $x_i>0$ for all $i<\frac{\pi}{2\underline\zeta_{k-1}}+1/2$. The explicit formula for $\hat{x}_i$ from Lemma \ref{lem:rec} in combination with \eqref{FhatF} now leads to the lower bound.

The upper bound is obtained similarly by estimating $\varepsilon_i$ from above
\[
F_1=1-\varepsilon_{1}
\leq 1-\overline\varepsilon_{k-1}=:\overline F_1
\]
where $\overline F_i$ is for $i\geq 2$ given by
\[
\overline F_i:=2-\overline\varepsilon_{k-1}-\overline F_{i-1}^{-1}.
\]
By induction we find for $i\geq 1$ as long as $F_{i-1}\geq 0$ that
\[
F_i=2-\varepsilon_{i}-F_{i-1}^{-1}\leq 2-\lambda -\overline\varepsilon_{k-1}-\overline F_{i-1}^{-1}=\overline F_i.
\]
By the definition of $\overline F_i$, we see from Lemma \ref{lem:rec} with  $\varepsilon=\overline\varepsilon_{k-1}$ that the upper bound holds.
\end{proof}

An obstacle of Theorem \ref{fiedlerbounds} is that one has to know the eigenvalue $\lambda$ and the functions $f_{T_i}$.
We introduce
\begin{align}
\begin{split}
\label{ohnef}
\underline\varepsilon_{k}(\lambda):=\min_{i=1,\ldots,k}\frac{|V(T_i)|\lambda-\|L_{v_i}(T_i)\|^{-1}\lambda^2}{1-\|L_{v_i}(T_i)\|^{-1}\lambda},\\
\bar \varepsilon_k(\lambda):=\max_{i=1,\ldots,k}\frac{|V(T_i)|\lambda-\|L_{v_i}(T_i)^{-1}\|\lambda^2}{1-\|L_{v_i}(T_i)^{-1}\|\lambda}.
\end{split}
\end{align}

The corollary below follows from Theorem \ref{fiedlerbounds} and the bounds for $f_{T_i}$ from Proposition \ref{fprop} (d) applied to the right hand sides of \eqref{annlap}. This implies that one can choose $\bar\zeta_k$ and $\underline\zeta_k$ in such a way that the bounds no longer depend on the functions $f_{T_i}$.
\begin{corollary}
\label{cor:last}
Let $T$ be a tree with a path $v_1\ldots v_k$ and trees $T_1,\ldots,T_k$ as in Figure \ref{dergraph}. Assume that $\lambda\in\sigma(L(T))$ satisfies
\begin{align}
\label{smalllambda}
\lambda<\min\limits_{i=1,\ldots,k-1}\lambda_{\min}(L_{v_i}(T_i))
\end{align}
and $0\leq \underline\varepsilon_{k-1}(\lambda)\leq\bar\varepsilon_{k-1}(\lambda)\leq 2$.

Then $\lambda$ is a simple eigenvalue with eigenvector $x$ and the bounds on the ratios of entries of $x$, \eqref{main1} and \eqref{main2} hold with
\[
\underline{\zeta}_{k-1}=\arccos\left(1-\frac{1}{2}\bar\varepsilon_{k-1}(\lambda)\right), \quad
\overline{\zeta}_{k-1}=\arccos\left(1-\frac{1}{2}\underline\varepsilon_{k-1}(\lambda)\right).
\]

\end{corollary}
\begin{remark}
Assuming that $\lambda=a(T)$ one can use the bounds
\begin{align}
\label{atbounds}
2-2\cos\left(\frac{\pi}{n+1}\right)\leq a(T)\leq2-2\cos\left(\frac{\pi}{d(T)+1}\right)
\end{align}
from \cite[p.\ 187]{Cvetkovic1980} and \cite{Fiedler1973}, to see that the assumption \eqref{smalllambda} holds for sufficiently large diameters of $T_k$.
Using the monotonicity of the right hand sides of \eqref{ohnef} as a function of $\lambda$ in combination with \eqref{atbounds} one has  expressions for $\bar\zeta_{k-1}$ and $\underline\zeta_{k-1}$ that also do not depend on $a(T)$.

Let $T$ be a path then let $T_i$ be given by  $V(T_i)=\{v_i\}$ then  $\underline{\varepsilon}_k(\lambda)=\overline{\varepsilon}_k(\lambda)=\lambda$ and the bounds in Theorem \ref{fiedlerbounds} and Corollary \ref{cor:last} hold with equality.
\end{remark}

\begin{example}
Let us consider caterpillar trees and we assume that the central path is the unique longest path (see Figure \ref{fig:caterpillar}). By Corollary \ref{cor:cater} the extremal values of the Fiedler vector are attained at $v_1$ and $v_k$. Furthermore, each $T_i$ is a star graph on $m_i+1$ vertices while its central vertex $v_i$ lies on the central path. It is easy to see that in this case $L_{v_i}(T_i)=I_{m_i}$ and each entry of $f_i$ is one. Hence, we have $f_{T_i}(\lambda) = f_i^T(L_{v_i}(T_i)-\lambda I_{m_i})^{-1}f_i=\frac{m_i}{1-\lambda}$ and so
\begin{align*}
f_{T_i}(\lambda)-\deg_{T_i}(v_i)&=\frac{m_i}{1-\lambda}-m_i=\frac{\lambda}{1-\lambda}m_i.
\end{align*}
If the caterpillar is nontrivial, i.e.\ not a simple path, we must have by the assumption that the central path is the unique longest path that $k\geq 3$. This implies $a(T)<1$ and by the same assumption we have $\min_{i=1,\ldots,\ell}m_i=0$ for all $l\in\N$ with $l\leq k$. This yields the representations from Corollary \ref{cor:last}
\begin{align*}
\min_{i=1,\ldots,\ell}(f_{T_i}(\lambda)-\deg_{T_i}(v_i))&=0\\
\max_{i=1,\ldots,\ell}(f_{T_i}(\lambda)-\deg_{T_i}(v_i))&\leq\frac{\lambda}{1-\lambda}\max_{i=1,\ldots,\ell}m_i.
\end{align*}
\end{example}

\begin{figure}
    \centering
    \includegraphics[scale=0.7]{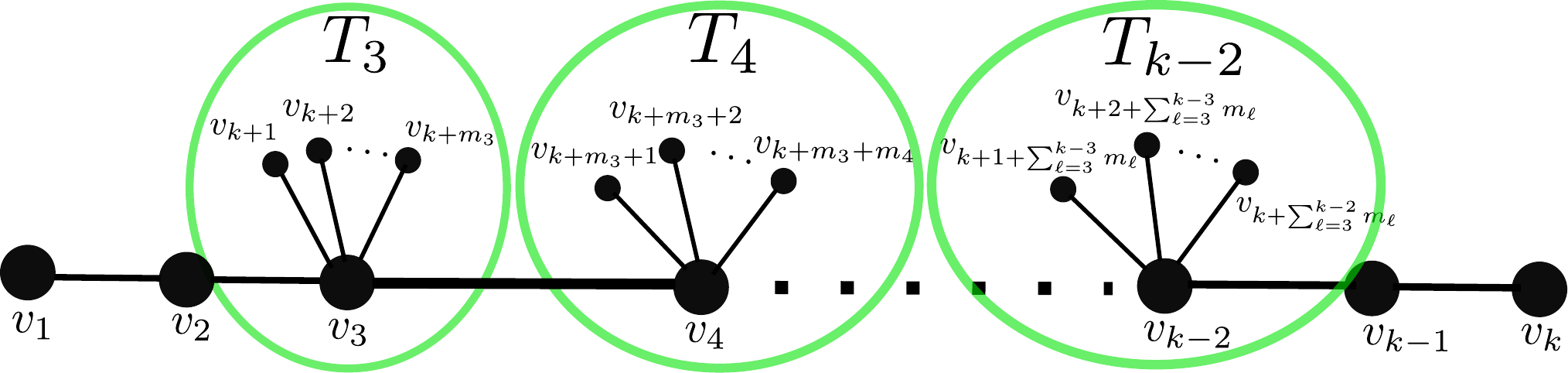}
    \caption{Caterpillar graphs with a central path being the unique longest path. 
    }
    \label{fig:caterpillar}
\end{figure}

\section{Local extrema of eigenvectors}
\label{sec:locExtr}

For paths $v_1\ldots v_k$ in a tree $T$ where $v_1$ is a pendant vertex, we have seen in Section \ref{sec:ratios} that the eigenvector entry $x_1>0$ at $v_1$ is larger than the value at $v_2$ for sufficiently small eigenvalues.
Here we provide a condition for sufficiently small eigenvalues of $L(T)$ where one can check if $v_1$ is also extreme among all vertices in $T_{i}$ for all $2\leq i\leq k$.

We introduce the so-called $L$-configurations. The idea is to compare the entries of the eigenvector at the vertices $v_i$ of the fixed path with the entries of the eigenvector at vertices in the attached trees $T_i$. We split up the tree $T_i$ into $\deg_{T_i}(v_i)$ disjoint subtrees $T_{i,j}$.
For a given path in $T_i$ starting at $v_i$ this path is in one tree $T_{i,j}$. Denote by $v_{i,j}^{(l)}=v_i,\ldots,v_{i,j}^{(1)}$ the vertices on this path and to each vertex $v_{i,j}^{(m)}$ we consider the attached tree $T_{i,j}^{(m)}$ as in Figure \ref{dergraph}. Then we apply the Schur reduction on the subpath $v_1\ldots v_{i-1}$ and on the subpath $v_{i,j}^{(1)}\ldots v_{i,j}^{(l)}$.

\begin{figure}[htbp!]
    \centering
    \includegraphics[scale=0.6]{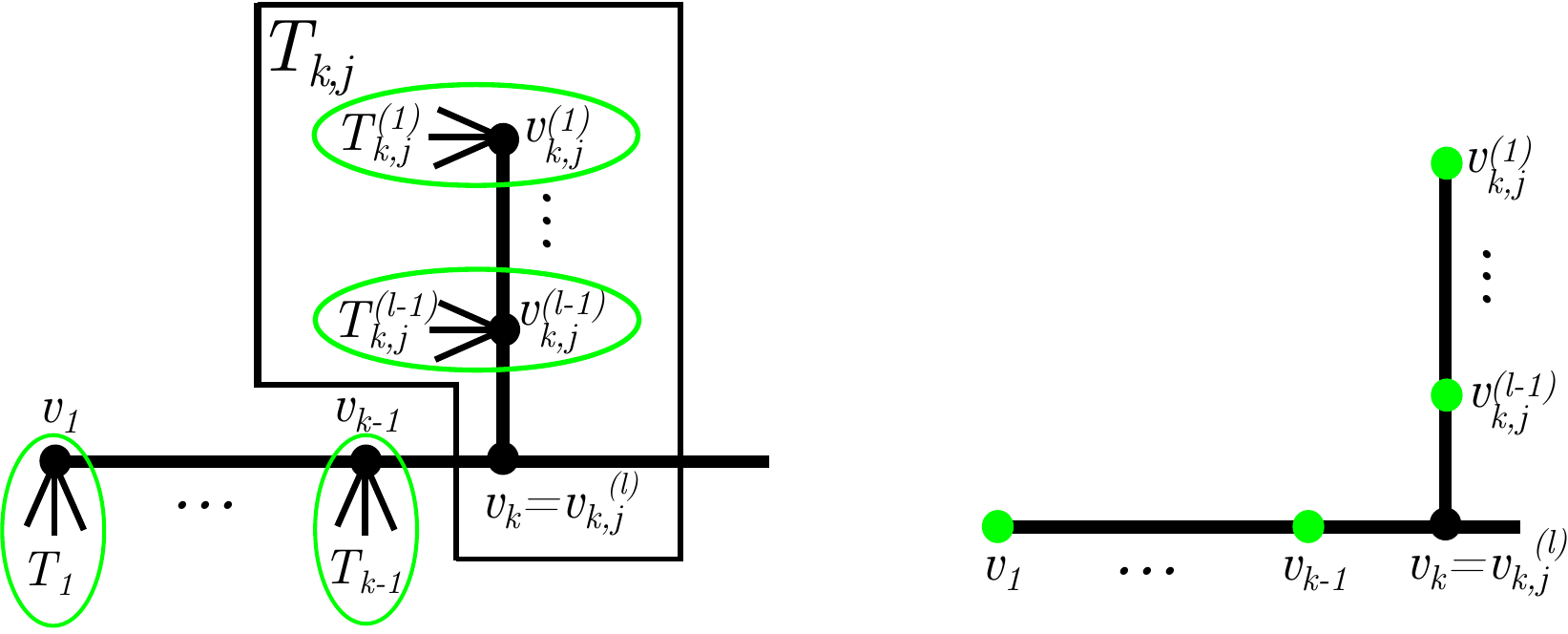}
    \caption{The figure shows an $L$-configuration, where we compare the entries of the eigenvector on the path $v_1\ldots v_{k-1}$ with the entries on the paths in $T_{k,j}$ for all $j=1,\ldots,\deg_{T_k}(v_k)$.
    The Schur reduction was applied only at the green colored vertices.}
    \label{lconfig}
\end{figure}

In the proposition below we show that under certain assumptions, the characteristic value at $v_1$ is maximal among all values at the green colored vertices in Figure \ref{lconfig}.

\begin{proposition}

Let $\lambda$ be an eigenvalue of $L(T)$ with eigenvector entry $x_1>0$ such that \eqref{annlap} holds for some $\underline\zeta_{k-1}\in(0,\frac{\pi}{2})$ with $k\leq \frac{\pi}{\underline\zeta_{k-1}}+\frac{1}{2}$. Consider the setting in Figure \ref{lconfig} with vertices $v_1,\ldots,v_{k}$ of the lower path and vertices $v_{k,j}^{(1)},\ldots,v_{k,j}^{(l)}=v_{k}$ with $l\leq k$ on the vertical path at $v_k$. If for all $i=1,\ldots,l$ and  $j=1,\ldots,\deg_{T_k}(v_k)$,  $T_{k,j}^{(i)}$ can be obtained from of $T_i$ by removing pendant vertices 

then we have for all $i=1,\ldots,l$ that
\begin{align}
\label{maxentry}
x_1\geq x_{k,j}^{(i)}\geq 0.
\end{align}

If $T_{k,j}^{(i)}$ is a proper subtree of $T_i$ for some index $i_0$ then \eqref{maxentry} is also strict for all $i_0\leq i\leq l$.
Furthermore, if $d(T_{k,j})\leq k-1$ for all  $j=1,\ldots,\deg_{T_k}(v_k)$ then \eqref{maxentry} holds for all vertices of $T_k$.
\end{proposition}
\begin{proof}
We consider the following two sequences
\begin{align*}
F_1^{|}(\lambda):=s_{T_{k,j}^{(1)}}(\lambda), \quad F_i^{|}(\lambda):=s_{T_{k,j}^{(i)}}(\lambda)-F_{i-1}^|(\lambda)^{-1},\ i\geq 2
\end{align*}
and
\begin{align*}
F_1^{-}(\lambda):=s_{T_{1}}(\lambda),\quad F_i^{-}(\lambda):=s_{T_{i}}(\lambda)-F_{i-1}^-(\lambda)^{-1},\ i\geq 2.
\end{align*}
Since the kernel of the matrix \eqref{tridiag} is the eigenvector at $\lambda$ restricted to the path, it is easy to see from the equation $S_{T_1,\ldots,T_k}(\lambda)x=0$ and Theorem \ref{fiedlerbounds} that
\begin{align}
\label{fistEV}
F_i^{-}(\lambda)=\frac{x_{i+1}}{x_{i}}>0,\quad F_i^{|}(\lambda)=\frac{x_{k,j}^{(i+1)}}{x_{k,j}^{(i)}}>0.
\end{align}

The assumption that $T_{k,j}^{(1)}$ is a subtree of $T_1$ with Proposition \ref{fprop} (e) implies that
\[
0< F_1^-(\lambda)\leq F_1^|(\lambda).
\]
We estimate for $2\leq i\leq l-1$ under the assumption that $F_{i-1}^-(\lambda)\leq F_{i-1}^|(\lambda)$ and with Proposition \ref{fprop} (e)
\begin{align*}
0<F_i^-(\lambda)&=\deg(v_i)-\lambda-f_{T_i}(\lambda)-F_{i-1}^-(\lambda)^{-1}\\ &=\deg(v_i)-\deg_{T_i}(v_i)-\lambda-\sum_{m=1}^{\infty}\frac{f_{T_i}^{(m)}(0)\lambda^m}{m!}-F_{i-1}^-(\lambda)^{-1}\\
&\leq\deg(v_{k,j}^{(i)})-\deg_{T_{k,j}^{(i)}}(v_{k,j}^{(i)})-\lambda-\sum_{m=1}^{\infty}\frac{f_{T_{k,j}^{(i)}}^{(m)}(0)\lambda^m}{m!}-F_{i-1}^-(\lambda)^{-1}\\
&\leq \deg(v_{k}^{(i)})-\lambda-f_{T_{k}^i}(\lambda)-F_{i-1}^|(\lambda)^{-1}\\ &=F_i^|(\lambda).
\end{align*}
Since both paths have a joint vertex at $v_{k}$, we have from the representation \eqref{fistEV} that
\[
\prod_{i=1}^{k-1}F_i^-(\lambda) x_1=x_{k}=\prod_{i=1}^{l-1}F_i^|(\lambda) x_{k,j}^{(1)}
\]
and since $0<F_i(\lambda)<1$,
\[
\frac{x_1}{x_{k,j}^{(1)}}=\frac{\prod_{i=1}^{l-1}F_i^|(\lambda)}{\prod_{i=1}^{k-1}F_i^-(\lambda) }\geq \frac{\prod_{i=1}^{l-1}F_i^|(\lambda)}{\prod_{i=1}^{l-1}F_i^-(\lambda) }\geq 1.
\]
Furthermore, $x_{k,j}^{(1)}\geq x_{k,j}^{(i)}$ for $i=1,\ldots,l$ implies \eqref{maxentry}. Assume that $d(T_{k,j})\leq k-1$, then for all vertices $v\in V(T_{k,j})$ there exists a path $v=v_{k,j}^{(1)}\ldots v_{k,j}^{(l)}=v_k$ and therefore \eqref{maxentry} holds for all vertices of $T_{k,j}$.
\end{proof}

\section{Discussion}\label{sec:discussion}
In this article, we have investigated the structure of Fiedler vectors of trees. One of our main tools is the Schur reduction introduced in Section \ref{sec:schur}. We remark that the reduction from equation \eqref{tridiag} can be extended to graphs where the $T_i$ are arbitrary graphs instead of trees. In both cases, applying the matrix-tree-theorem to a slightly modified graph, the resolvents in equation \eqref{def:fTi} can be computed combinatorically by counting spanning 2- and 3-forests. This enables one to relate the entries of the Fiedler vector to the graph topology in more detail and will be the topic of a future investigation.\\
Taylor expansion of ratios of Laplacian eigenvector entries then allows us to restrict the location of the extremal entries of the Fiedler vector. As an application we introduce caterpillar trees and a class of generalized caterpillars. We remark that not only can this approach be applied to other classes of trees, but due to the generality of Lemma \ref{thm:pendant}, it can also be applied to other eigenvalues than the algebraic connectivity $a(T)$. Furthermore, we derive a sufficient criterion for the Fiedler rose to have the (FED) property and we show that the mechanism which destroys this property in the Fiedler rose can be generalized to a large class of trees $T$. More precisley, when gluing a star graph with sufficiently many leaves to $T$, one extremal entry of the Fiedler vector will lie on the star and therefore the (FED) property is not preserved.\\
Finally, we use the Schur reduction in order to bound ratios of eigenvector entries with applications to the Fiedler vector and introduce a large class of trees in which we can identify a local extremal value of a Laplacian eigenvector.\\
Although we have given partial answers for large classes of trees and identified other classes of trees for which the (FED) property holds true, it remains an open problem to identify the largest class of trees which possess the (FED) property.

\bibliographystyle{plain}
\bibliography{GP18_arxiv}

\begin{thebibliography}{10}

\bibitem{Abreu2007}
N.~Abreu.
\newblock Old and new results on algebraic connectivity of graphs.
\newblock {\em Linear Algebra Appl.}, 423:53--73, 2007.

\bibitem{Abreu2016}
N.~Abreu, L.~Markenzon, L.~Lee, and O.~Rojo.
\newblock On trees with maximum algebraic connectivity.
\newblock {\em Appl. Anal. Discrete Math.}, 10:88--101, 2016.

\bibitem{Aitken1939}
A.~C. Aitken.
\newblock {\em Determinants and matrices}.
\newblock University Mathematical Texts, Oliver \& Boyd, Edinburgh, 1939.

\bibitem{AndradeDahl2017}
E.~Andrade and G.~Dahl.
\newblock Combinatorial {P}erron values of trees and bottleneck matrices.
\newblock {\em Linear and Multilinear Algebra}, 65(12):2387--2405, 2017.

\bibitem{Banuelos1999}
R.~Ba{\~n}uelos and K.~Burdzy.
\newblock On the "hot spots" conjecture of {J}. {R}auch.
\newblock {\em J. Funct. Anal.}, 164:1--33, 1999.

\bibitem{Barooah2006}
P.~Barooah and J.~P. Hespanha.
\newblock Graph effective resistance and distributed control: Spectral
  properties and applications.
\newblock In {\em Proceedings of the 45th IEEE Conference on Decision and
  Control}, pages 3479--3485, Dec 2006.

\bibitem{Bartlett1951}
M.~Bartlett.
\newblock An inverse matrix adjustment arising in discriminant analysis.
\newblock {\em Ann. Math. Statist.}, 22:107--111, 1951.

\bibitem{Biyikoglu2007}
T.~B{\i}y{\i}ko{\u{g}}lu, J.~Leydold, and P.F. Stadler.
\newblock {\em Laplacian Eigenvectors of Graphs}.
\newblock Springer, 2007.

\bibitem{Burdzy1999}
K.~Burdzy and W.~Werner.
\newblock A counterexample to the "hot spots" conjecture.
\newblock {\em Ann. Math.}, 149:1:309--317, 1999.

\bibitem{Chen2012}
J.~Chen, J.~Lu, C.~Zhan, and G.~Chen.
\newblock {\em Handbook of Optimization in Complex Networks}, chapter~4, pages
  81--113.
\newblock Springer, 2012.

\bibitem{Chung2011}
M.K. Chung, S.~Seo, N.~Adluru, and H.K. Vorperian.
\newblock Hot spots conjecture and its application to modeling tubular
  structures.
\newblock In {\em Machine Learning in Medical Imaging}, pages 225–--232.
  Springer, 2011.

\bibitem{Cvetkovic1980}
D.~Cvetkovi\'{c}, M.~Doob, and H.~Sachs.
\newblock {\em Spectra of Graphs}.
\newblock Academic Press, 1980.

\bibitem{Davis1970}
C.~Davis and W.~M. Kahan.
\newblock The rotation of eigenvectors by a perturbation. {I}{I}{I}.
\newblock {\em SIAM J. Numer. Anal.}, 7:1--46, 1970.

\bibitem{Doerfler2012a}
F.~D{\"o}rfler and F.~Bullo.
\newblock Kron reduction of graphs with applications to electrical networks.
\newblock {\em IEEE Trans. Circuits Syst.}, pages 150--163, 2012.

\bibitem{El-Basil1987}
Sherif El-Basil.
\newblock Applications of caterpillar trees in chemistry and physics.
\newblock {\em J. Math. Chem.}, 1(2):153--174, Jul 1987.

\bibitem{Estrada2011}
E.~Estrada.
\newblock Community detection based on network communicability.
\newblock {\em Chaos}, 21(1), 2011.

\bibitem{Evans2011}
L.~C. Evans.
\newblock The {F}iedler {R}ose: On the extreme points of the {F}iedler vector.
\newblock {\em arXiv:1112.6323}, 2013.

\bibitem{Fiedler1973}
M.~Fiedler.
\newblock Algebraic connectivity of graphs.
\newblock {\em Czechoslovak Math. J.}, 23:298--305, 1973.

\bibitem{Fiedler1975}
M.~Fiedler.
\newblock A property of eigenvectors of nonnegative symmetric matrices and its
  application to graph theory.
\newblock {\em Czechoslovak Math. J.}, 25:619--633, 1975.

\bibitem{Fiedler1962}
M.~Fiedler and V.~Pt\'{a}k.
\newblock On matrices with non-positive off diagonal entries and positive
  principal minors.
\newblock {\em Czechoslovak Math. J.}, 12:382--400, 1962.

\bibitem{Grone1990}
R.~Grone, R.~Merris, and V.~S. Sunder.
\newblock The {L}aplacian spectrum of a graph.
\newblock {\em SIAM J. Matrix Anal. Appl.}, 11:2:218--238, 1990.

\bibitem{Guo2006}
J.-M. Guo.
\newblock The \textit{k}th {L}aplacian eigenvalue of a tree.
\newblock {\em J. Graph Theory}, 54:1:51--57, 2006.

\bibitem{Harary1973}
Frank Harary and Allen~J. Schwenk.
\newblock The number of caterpillars.
\newblock {\em Discr. Math.}, 6(4):359 -- 365, 1973.

\bibitem{Horn2013}
R.~A. Horn and C.~R. Johnson.
\newblock {\em Matrix analysis (Second edition)}.
\newblock Cambridge Univ. Press, 2013.

\bibitem{Kirkland1996}
S.~Kirkland and B.~Shader.
\newblock Characteristic vertices of weighted trees via {P}erron values.
\newblock {\em Linear and Multilinear Algebra}, 40:311--325, 1996.

\bibitem{Lefevre2013}
J.~Lef\`{e}vre.
\newblock Fiedler vectors and elongation of graphs: A threshold phenomenon on a
  particular class of trees.
\newblock {\em arXiv:1302.1266}, 2013.

\bibitem{Lovasz1993}
L.~Lov{\'a}sz.
\newblock Random walks on graphs.
\newblock {\em Combinatorics, {P}aul {E}rd{\"o}s is eighty}, 2:1--46, 1993.

\bibitem{Merris1998}
R.~Merris.
\newblock Laplacian graph eigenvectors.
\newblock {\em Linear Algebra Appl.}, 278:221--236, 1998.

\bibitem{Miekkala1993}
Ulla Miekkala.
\newblock Graph properties for splitting with grounded {L}aplacian matrices.
\newblock {\em BIT Numerical Mathematics}, 33(3):485--495, Sep 1993.

\bibitem{Mohar1992}
B.\ Mohar.
\newblock Laplace eigenvalues of a graph - a survey.
\newblock {\em Discrete Math.}, 109:171--183, 1992.

\bibitem{Molitierno2012}
J.~J. Molitierno.
\newblock {\em Applications of Combinatorial Matrix Theory to Laplacian
  Matrices of Graphs}.
\newblock CRC Press, 2012.

\bibitem{Nakatsukasa2013}
Y.~Nakatsukasa, N.~Saito, and E.~Woei.
\newblock Mysteries around the graph {L}aplacian eigenvalue 4.
\newblock {\em Linear Algebra Appl.}, 438:8:3231--3246, 2013.

\bibitem{Pade2016}
J.~P. Pade.
\newblock {\em Synchrony and Bifurcations in Coupled Dynamical Systems and
  Effects of Time Delay}.
\newblock PhD thesis, Institute of Mathematics, HU Berlin, Berlin, Germany,
  2016.

\bibitem{Pade2015}
J.~P. Pade and T.~Pereira.
\newblock Improving the network structure can lead to functional failures.
\newblock {\em Nature Scient. Rep.}, 5, 2015.

\bibitem{Pirani2014}
M.~Pirani and S.~Sundaram.
\newblock Spectral properties of the grounded {L}aplacian matrix with
  applications to consensus in the presence of stubborn agents.
\newblock In {\em 2014 American Control Conference}, pages 2160--2165, June
  2014.

\bibitem{Shen2010}
X.~Shen, X.~Papademetris, and R.~T. Constable.
\newblock Graph-theory based parcellation of functional subunits in the brain
  from resting-state f{MRI} data.
\newblock {\em NeuroImage}, 50:3:1027--1035, 2010.

\bibitem{Griffing2009}
E.~Stone and A.~Griffing.
\newblock On the {F}iedler vectors of graphs that arise from trees by {S}chur
  complementation of the {L}aplacian.
\newblock {\em Linear Algebra Appl.}, 431:1869--1880, 2009.

\bibitem{Griffing2013}
E.~Stone, B.~Lynch, and A.~Griffing.
\newblock An eigenvector interlacing property of graphs that arise from trees
  by {S}chur complementation of the {L}aplacian.
\newblock {\em Linear Algebra Appl.}, 438:1078--1094, 2013.

\bibitem{Zhang2011}
X.-D-Zhang.
\newblock The {L}aplacian eigenvalues of graphs: a survey.
\newblock {\em arXiv:111.2897}, 2011.

\bibitem{Xia2017}
Weiguo Xia and Ming Cao.
\newblock Analysis and applications of spectral properties of grounded
  {L}aplacian matrices for directed networks.
\newblock {\em Automatica}, 80:10 -- 16, 2017.

\bibitem{Zhang2005}
F.~Zhang.
\newblock {\em The Schur Complement and its Applications}.
\newblock Numerical Methods and Algorithms, Springer, New York, 2005.

\end{thebibliography}

\end{document}